\documentclass{tac}

\usepackage{amsmath,amssymb}

\usepackage[all]{xy}  
\def\xybiglabels{\def\labelstyle{\textstyle}} 

\xyoption{2cell}
\UseAllTwocells

\usepackage[colorlinks=true]{hyperref}
\hypersetup{allcolors=[rgb]{0.1,0.1,0.4}}

\author{Jeffrey C.\ Morton and Roger Picken}

\thanks{The authors would like to acknowledge Dany Majard for many helpful
discussions in the writing of this paper. This work was supported in part by
the Graduiertenkolleg 1670 ``Mathematics Inspired by String Theory and QFT'' of 
the Mathematics Department of the University of Hamburg, and by the
Funda\c{c}\~{a}o para a Ci\^{e}ncia e a Tecnologia, Portugal, through projects
PTDC/MAT/101503/2008, EXCL/MAT-GEO/0222/2012 and 
 UID/MAT/04459/2013.}

\address{Department of Mathematics and Statistics,\\
The University of Toledo, 2801 W. Bancroft St. \\
Toledo, OH 43606-3390, USA\\[5pt]
 Center for Mathematical Analysis, Geometry and Dynamical Systems,\\
Instituto Superior T\'ecnico, Universidade de Lisboa, \\
Av. Rovisco Pais, 1049-001 Lisboa, Portugal\\
}

\title {Transformation Double Categories Associated to 2-Group Actions}

\copyrightyear{2015}


\eaddress{jeffrey.c.morton@theoreticalatlas.net\CR roger.picken@tecnico.ulisboa.pt}

\newcommand{\defn}[1]{\textbf{#1}}

\newcommand{\ra}{\rightarrow}
\newcommand{\Ra}{\Rightarrow}

\newcommand{\opname}[1]{\operatorname{#1}}
\newcommand{\cat}[1]{\boldsymbol{\opname{#1}}}
\newcommand{\squaremor}[1]{\begin{array}{|lcr|}\hline & {#1} & \\
    \hline\end{array}}

\newcommand{\wquot}{/\!\!/}
\newcommand{\G}{\mathcal{G}}
\newcommand{\C}{\cat{C}}
\newcommand{\act}{\blacktriangleright}

\begin{document}

\maketitle
\begin{abstract}
 Transformation
  groupoids associated to group actions capture the interplay between
  global and local symmetries of structures described in set-theoretic
  terms. This paper examines the analogous situation for structures
  described in category-theoretic terms, where symmetry is expressed
  as the action of a 2-group $\G$ (equivalently, a categorical group)
  on a category $\C$. It describes the construction of a
  transformation groupoid in diagrammatic terms, and considers this
  construction internal to $\cat{Cat}$, the category of
  categories. The result is a double category $\C \wquot \G$ which
  describes the local symmetries of $\C$. We define this and describe
  some of its structure, with the adjoint action of $\G$ on itself as
  a guiding example.
\end{abstract}

\section{Introduction}

Symmetry is a fundamental notion both in mathematics and in
physics.  In this paper, we consider the symmetries of categories: 
in particular, if a symmetry of a category $\cat{C}$ is represented as an 
endofunctor of $\cat{C}$, one also has 'symmetries of symmetries', because 
there are natural transformations between these functors. These two 
levels of symmetry can be represented by using a 2-group. Thus we consider 
actions of 2-groups on categories, and develop the analog of the 
``transformation groupoid'' construction, which captures the local 
view of the symmetries for a group action on a set.

\subsection{Symmetries, group actions and 2-group actions}

The usual mathematical representation of symmetry is as a
\textit{group action} on a set, or perhaps on a manifold, vector
space, or whatever sort of structure is of interest. 
Not only the objects on which the groups act, but also the groups 
themselves may have various sorts of structure  
- one may be interested in algebraic groups, or Lie groups, for
example. In particular, symmetry in this sense is described by the
action of a \textit{group object} in an appropriate category: an
object in the category equipped with structure maps, such as the 
multiplication map $m : G \times G \ra G$, satisfying the axioms which 
define a group.

Even superficially very different constructions, such as quantum
groups, fit this pattern to some extent. A quantum group is a rather broad
concept with many variations, but the starting point is the idea of a
Hopf algebra, a bialgebra with a certain structure. Quantum groups
play a role in describing symmetry in noncommutative geometry and
other algebraic settings. Hopf algebras may be seen as group objects
in $\cat{Alg}^{op}$, the opposite category of some category of
algebras. (Which category of algebras depends on the context within
noncommutative geometry).

The notion of a group action, also in the broader sense above, 
describes what are known as \textit{global} symmetries:
 global operations on the object which supports the symmetry,
leaving it in some sense unchanged. The exact sense in which it is
to be unchanged, and what a transformation of the object can consist
of, is the guiding criterion in finding its group of symmetries. On
the other hand, there is also a \textit{local} concept of symmetry.
Following Weinstein \cite{weinstein-symmetry}, local symmetry may be 
described in terms of \textit{groupoids}. 
Recall that groupoids are categories in which all morphisms
are invertible. Now the entire set which supports the symmetry, such 
as the set of points of a space, forms the objects of the groupoid, and
the morphisms of the groupoid represent symmetry relations which relate 
one point to another, making them ``indistinguishable'' under the 
relation in question. Such symmetry relations may or may not arise from a
global group action on the set, and thus the local ``groupoid approach'' 
to symmetry is more fine-grained than the global ``group action approach''.

Of course, given a global group action, it can also be viewed from an 
equivalent local 
perspective, in terms of the \textit{transformation groupoid} associated to 
the global
group action. Note that not every groupoid representing local symmetry 
is a transformation groupoid\footnote{Some easy counter-examples appear as 
sub-groupoids of transformation groupoids. For instance, suppose we have 
an action of a Lie group $G$ on a 
manifold $M$. We can 
consider a small neighborhood $U$ in $M$, and take the sub-groupoid of 
the transformation groupoid associated to the action of $G$ with only 
those objects corresponding to points in $U$. This has all the 
morphisms taking points in $U$ to other points in $U$, even if $U$ does 
not contain any entire orbit of the $G$ action. That is, it retains 
all the 'local' information about $M$'s symmetries that can be 
detected in $U$.}. In our generalization of group actions to 2-group 
actions, we will develop both global and local perspectives, and show 
them to be equivalent.

The construction of a transformation groupoid, coming from a group action 
on a set, makes sense not only in the category of sets, but also in
other contexts. Typical examples would include a Lie group acting on a
manifold, a topological group acting on a topological space, or an
algebraic group acting on a variety. In reasonable situations (which
always apply in a topos, such as the category of sets, or of
topological spaces), one can construct a transformation groupoid which 
captures the
local picture implied by such global symmetry groups. These will be,
respectively, Lie groupoids, topological groupoids, or algebraic
groupoids in the cases mentioned above. 
We can say that each of these situations is an instance of a general
process, which can be described diagrammatically, and independently of the
type of objects playing the role of ``sets''. In other words, there is a
general construction of a transformation groupoid, which makes sense
in a category (of a suitable kind). Each of the cases we mentioned are
examples of this construction ``internal to'' a chosen ambient
category, such as topological spaces, manifolds, varieties, and so
forth.

The situation we describe in this paper generalizes these ideas in 
a different direction. We are interested in describing symmetry
of \textit{categories} by taking the notion of a group action on a set to
a higher categorical level, both for the group and for the set it acts on.
Thus we will be studying the action of a 2-group on a category.

In order to understand this step, it is useful to recall that a group is a
particular type of category, with just a single object and all morphisms 
invertible. The product on the group is the composition of morphisms.
Now, given an object $X$ in a
category, $\cat{X}$, there is a sub-category $End(X)$ of $\cat{X}$ 
with one object $X$ and morphisms consisting of all the endomorphisms of $X$. 
Restricting to the invertible such morphisms, one gets $Aut(X)$. This is a
subcategory of $\cat{X}$, but it is also a group in the sense just
mentioned. It is, indeed, the full symmetry group of the object $X$. In these terms,    
the action of a group $G$ on $X$ is described by a functor from the category 
$G$ to the category of automorphisms of $X$, say  $\phi:G \ra Aut(X)$. 

Now, just as a group may be regarded as a special sort of category, so
too a 2-group can be seen as a special kind of 2-category, namely one with
just a single object and invertible 1- and 2-morphisms.  
Given a category $\cat{C}$, i.e. an object in the 2-category $\cat{Cat}$, one has
endofunctors from $\cat{C}$ to itself, but also natural
transformations between such endofunctors. So $End(\cat{C})$ is a
2-category with a single object $\cat{C}$, whose 1-morphisms are endofunctors of 
$\cat{C}$  and whose 2-morphisms are natural transformations between endofunctors. 
If we restrict to only those endofunctors and natural 
transformations which are invertible,
we have $Aut(\cat{C})$, which is a 2-group. It
is, indeed, the full symmetry 2-group of $\cat{C}$. Then, in complete analogy with
the action of a group $G$ on a set $X$, the action of a 
2-group $\G$ on a category $\cat{C}$ is described by a 2-functor from the 
2-category 
$\G$ to the 2-category of automorphisms of $C$:
$$
\Phi:\G \ra Aut(\cat{C}).
$$
Intuitively, following Baez and Lauda in \cite{hda5}, we are considering not only symmetries of
$\cat{C}$ (the 1-morphism level of $\Phi$), but also ``symmetries between symmetries'' (the 
2-morphism level of $\Phi$).

As an aside, we note that there are other approaches to generalizing the action of $G$ on $X$ to higher
settings. Some recent work has been done on categories equipped
with actions of a group $G$, (note, not a 2-group) and their ``equivariantizations'' (see,
for example, \cite{kirillov, naidu, burciu-natale} among
others). These equivariantizations generalize the notion of the fixed 
points of a group
action on a (possibly structured) set. From our perspective this is 
a rather
special case of the most general sense of symmetry for categories since we expect to have
not only morphisms between categories (namely, functors), but also morphisms between these morphisms
(namely, natural transformations). This means that, for us, the notion of
``symmetry'' of a category is inherently somewhat more subtle than that 
for a set, or
even for any set with an additional structure, such as a topological
space.
 Another generalization, going beyond ours, is to
consider the action of 2-groupoids, or yet more general categorical structures, 
on a category \cite{brown-mackenzie}. A 2-groupoid is a 2-category with invertible 1- and 2-morphisms, but 
it may have more than one object. However, we wish to focus on the 2-group case, because it naturally 
generalizes the important and ubiquitous notion of actions of groups, and because the application 
we have in mind (in higher gauge theory, to be described shortly) motivated us to study 2-group actions,
 and not the action of anything more general. Finally, we note the recent work by J. Elgueta \cite{elgueta},
where a weak 2-group acts on a groupoid by equivalences, rather than by endofunctors. Here the self-equivalences of 
an object in a 2-category naturally constitute a weak 2-group.
This is not directly 
comparable to our approach, since weak 2-groups are used throughout in Elgueta's paper.

Returning to our main discussion, the essential point in this paper is that 2-group actions, too, can be
understood in terms of an internal construction like those described above, i.e. transformation 
groupoids internal to some category. The ambient category
is now $\cat{Cat}$, the category whose objects are (small) categories and whose morphisms are functors. 
Then we can use the fact that an alternative way to
see 2-groups is as categorical groups, namely group objects in
$\cat{Cat}$. Starting with an ordinary group action on a set, the corresponding transformation groupoid
is obtained from a function, $\hat{\phi} : G \times X \ra X$, satisfying the usual properties for an action.
Here $G$ is treated as a set, although the properties involve the group operation on $G$. In the same way, 
a functor:
$$
\hat{\Phi}:\G\times \C \rightarrow \C 
$$
satisfying analogous properties to $\hat{\phi}$, and where $\G$ is treated as a category as opposed to a 2-category,
is equivalent to an action of the 2-group $\G$ on the category $\C$, as described above 
(Theorem \ref{thm:phifunctorial}). From $\hat{\Phi}$ one obtains the description of the action as 
a transformation groupoid internal to $\cat{Cat}$. In this way we get a local symmetry picture for the global action of a 2-group
on a category, just like the global and local pictures that we had for the action of an ordinary group.

We have been using two different understandings of a 2-group action on a category $\cat{C}$. The view where the 2-group is taken to be a group object in $\cat{Cat}$ is generally called ``internal''. By analogy, we may call the other ``external'', i.e. in terms of a 2-functor from the 2-category $\G$ to the 2-category $Aut(\cat{C})$, whose sole object is $\cat{C}$ . The ``internal'' description makes sense in any category where a group object can be defined, but describes a 2-group action only in $\cat{Cat}$. The ``external'' description always describes the action of a 2-group, by its very definition.
The existence of both internal and
``external'' views of symmetry for a category relies on the fact that
$\cat{Cat}$ is a \textit{closed} category, that is, a category with an
 ``internal $\opname{Hom}$'', so that the morphisms
from $A$ to $B$ form an object $\opname{Hom}(A,B)$ within the category.
Thus $Aut(\cat{C})$ can  be seen
as a category, equipped with certain structure maps which make it into a
2-group.

Now, the internal, local symmetry picture associated with $\hat{\Phi}$
above gives rise
to a transformation groupoid in the same ambient category as our group
object, and the object it acts on. So for categories, we find that
this local symmetry structure is most naturally encoded by a category
(indeed, it will be a groupoid) internal to $\cat{Cat}$ itself. Such a
structure is called $\cat{Cat}$-category, or often a \textit{double
  category} (though to avoid ambiguity, we will reserve this name for a
distinct but equivalent structure). 
In particular, this structure is not generally just a 2-category. An
external view of it shows that it has two different kinds of morphism,
denoted horizontal and vertical, and higher morphisms which naturally
have the shape of squares.\footnote{The shape of higher morphisms is a typical
issue in higher-categorical constructions: see e.g. \cite{chenglauda,
  leinster} for discussion of variants of higher categories.} We will show
in Section \ref{sec:transdoublecat} 
that there are several complementary perspectives for viewing this structure,
all highlighting particular features of the 2-group action. The groupoid 
internal to $\cat{Cat}$ perspective has a dual description as a category internal
to $\cat{Gpd}$ (the category of groupoids), and a third, symmetric perspective views
the structure as a double category (becoming a double groupoid when the acted-on 
category $\cat{C}$ is a groupoid). 
Thus our treatment here demonstrates that there is extra subtlety, through these
``variations in shape'', when we try to extend the concept of symmetry to higher
structures.  

\subsection{Motivation from higher gauge theory}

Our motivation for the present study comes from geometry, although this will play no
direct role in this paper. We will, however, describe a little of this
motivation here, and in a subsequent paper we will give more details
about the example which prompted us to study this question.

The example we had in mind comes from \textit{higher gauge theory} (HGT)
\cite{basch-hgt}. This is a general program of developing analogs to
concepts in gauge theory, for higher-categorical groups. It provides
generalizations of, among other things, the theory of connections on
bundles over manifolds. These analogs include the study of connections
on (non-abelian) gerbes \cite{giraud,murray-gerbes,martinspickeni,
martinspickenii,schreiberwaldorfii}. These entities have been the
study of considerable interest in recent years, particularly since
they can be used to describe theories in which one can naturally
describe the parallel transport of extended objects such as
(one-dimensional) strings or (higher-dimensional) ``membranes''.

Behind the current paper is a desire to better understand the
symmetries of the ``moduli spaces'' for such theories. That is, in
ordinary gauge theory, given a manifold $M$ and a Lie group $G$, there
is a space of all connections on a given principal $G$-bundle over
$M$. Furthermore, there is a group which acts on this space, the group
of all gauge transformations. There is a local picture of this
symmetry, in which both connections and gauge transformations are
described in terms of certain forms and functions on $M$. The relation
between the global symmetries and this local picture is
well-understood. The local picture describes a groupoid, whose objects
are the connections, coming from the global group action.

In the case of connections on gerbes, there is also a well-understood
local picture in terms of forms and functions. However, these
naturally form something more complex than a groupoid. The
relationship between this local picture and a notion of global
symmetries is less well-established. In our subsequent paper, we will
show a simple example of how the local picture arises from a global symmetry 2-group
acting on a category which plays the role of the moduli space of
connections on gerbes over $M$. For even higher-categorical
structures, known as $n$-gerbes, one can describe the transport of
higher-dimensional objects, or $n$-branes. This naturally has
connections to string theory \cite{schreiber-sati-stasheff,
  baez-perez}, and to other areas in theoretical physics.

Our main specific motivation within HGT relates to work 
of Yetter \cite{yettqft},
and Martins and Porter \cite{martinsporter}, extending a certain field
theory (the Dijkgraaf-Witten model) to 2-groups. One of the ultimate
goals in our project is to reproduce for these theories a
groupoid-based approach to such a model \cite{morton-etqft}, using the
double groupoids we find here, relating these field theories to the
representations of the double groupoids.

While the motivation from HGT is of interest to us, it is not necessary to
understand the present paper. For the moment, we are simply interested in
understanding better how the relationship between local and global
symmetries applies when the entity whose symmetries are involved
happens to be a category. Since the notion of symmetry is one of the
most fundamental, and categories have proven to be a vital and
important part of modern mathematics, the present study should be of interest well
beyond the context of our original example, and applicable in a variety of areas.

\subsection{Results and outline}

The main results of this work are structural theorems relating to the action of a strict
2-group $\G$ on a category $\C$, and the different ways in which this action may be described.

We begin in Section \ref{sec:2groups} with a discussion of (strict) 2-groups,
and various equivalent definitions which are useful in talking about
them. This includes the correspondence between 2-groups and crossed
modules and the definition of the double groupoid associated to 
the 2-group $\G$,
which is both an introduction to the notion of double category and a
useful 2-dimensional diagrammatic tool for 2-group calculations. 

 This
is followed by Section \ref{sec:2group-actions} which defines both the global and local
notions of a 2-group $\G$ acting on a category $\C$ (see the first part of the introduction), i.e. 
Definitions \ref{def:2grp_action} and \ref{def:2grp_action-Phi-hat}, and shows the equivalence 
between the two definitions in Theorem \ref{thm:phifunctorial}. This is followed in subsection 
\ref{adjoint} by a 
 basic example, namely the
adjoint action of a 2-group on itself, analogous to the adjoint action
of a group on itself by conjugation. (In the forthcoming paper on
higher gauge theory, we show that the adjoint action of $\G$ on itself 
describes higher gauge theory for the simple case when the manifold is the
circle, $S^1$.)

  In our main Section \ref{sec:transdoublecat},
we start by describing the construction of the transformation groupoid associated
to a group action on a set, in a diagrammatic form, which captures the
description of this groupoid in terms of elements. In 
subsection \ref{sec:transcat-cat}, repeating this
construction in $\cat{Cat}$ gives a groupoid internal to $\cat{Cat}$, 
as shown in Theorem \ref{thm:transf-gpd}. 
We develop a few different points of view on this internal groupoid
 in subsection \ref{sec:transdouble}, first taking 
a transposed internal view, giving rise to a category internal to $\cat{Gpd}$, 
the category of groupoids (Theorem \ref{thm:transposetransform}), and then 
giving a symmetrical definition 
in terms of a double category (Theorem \ref{thm:trans-dbl-cat}). The transposed viewpoint turns out
to have close connections to ordinary transformation groupoids for
certain group actions. Further aspects of these viewpoints are 
developed in subsection \ref{sec:struc-trans-double-cat}, and in particular we 
show that there are two 2-categories asociated to the transformation double category, 
namely the horizontal 2-category and the vertical 2-category, Theorems \ref{thm:horiztwocat}  and 
\ref{thm:verttwocat}. 
Our standard example, that of the adjoint
action of a 2-group on itself, gives a concrete case throughout, which
we develop in detail in \ref{sec:trans-adjoint}.

\section{2-Groups and 2-Group Actions on a Category}\label{sec:2groups}

In this paper, we will be describing certain actions of 2-groups. The
structures known as 2-groups are also sometimes called categorical
groups, $\mathbf{gr}$-categories or groupal groupoids. Moreover, all
of these can be shown to be equivalent to crossed modules (of
groups). As the abundance of terminology suggests, there are several
conceptually different, but logically equivalent, definitions which
can be given for these structures. We will primarily use the term
``2-group'', but in fact it will be useful for our discussion to be
able to move back and forth between the different definitions. In
Section \ref{sec:2gpdefns} we will recall those definitions which will
be used, and note how they are equivalent.

This prepares the ground for section \ref{sec:2group-actions}, where
 we will consider from two of these points of view how 2-groups
can act on categories, and the ``transformation'' structure which
results. This will be the analog of the transformation groupoid
associated to a group action on a set.

Note that in the following, to avoid cumbersome notation, 
 we use the notation $\cat{X}^{(0)}$ and $\cat{X}^{(1)}$ to
denote, respectively, the objects and morphisms of a category
$\cat{X}$, and similar notation for the object and morphism maps of a
functor. For a 2-category $\cat{X}$, when appropriate, we denote the 2-morphisms
by $\cat{X}^{(2)}$.

\subsection{2-Groups, Categorical Groups, and Crossed Modules}\label{sec:2gpdefns}

Here we lay out three definitions of equivalent structures: 2-groups,
categorical groups, and crossed modules. In this section, we will be
careful to distinguish the three, but in the rest of the paper we will
generally use the term ``2-group''. The main point of this section is
to highlight the well-known result that there is an equivalence
between the three definitions. 

\subsubsection{2-Groups as 2-Categories}\label{sec:2group-2cat}

One very useful definition of a 2-group is motivated by analogy to the
definition of a group as a (small) one-object category whose morphisms
are all invertible. In this case, the elements of the group are the
set of morphisms of the category, and the group multiplication is the
composition.

This highlights the way the definition of a category generalized the
notion of ``composition'' which begins with algebraic gadgets such as
groups. One might equally well define a (small) category in algebraic
language as a unital semigroupoid. Taking the definition of category
as more basic gives a group as a special case. While of course not the
standard definition of a group, it is at least straightforward to
generalize to ``higher'' groups in the sense of higher category
theory.

(We will assume some familiarity with 2-categories in what follows,
though for readers who are not so familiar we suggest the succinct
note by Leinster \cite{leinster-bb} as a starting point, and the more
comprehensive survey by Lack \cite{lack-bicat} and Chapter 7 of
Borceux \cite{borceux} for more detail.)

\begin{definition}
  A \defn{2-group} $\G$ is a 2-category with one object ($\G^{(0)} =
  \{ \star \}$), for which all 1-morphisms $\gamma \in \G^{(1)}$ and
  2-morphisms $\chi \in \G^{(2)}$ are invertible.
\end{definition}
Note that in the case of $\chi \in \G^{(2)}$, we require invertibility
with respect to both horizontal and vertical composition.

Let us unpack this definition more explicitly.

Suppose $\G$ is a 2-group. There are various 1-morphisms from $\star$
to itself, which have a composition operation, since all such
morphisms have matching source and target:
\begin{equation}
  \xymatrix{
    \star  & \star \ar[l]^{\gamma_1} & \star \ar[l]^{\gamma_2}
  } =
  \qquad
  \xymatrix{ 
    \star & \star \ar[l]^{\gamma_1 \circ \gamma_2}
  }
\end{equation}
Composition of morphisms is the multiplication, which is therefore
associative, and every morphism must have an inverse by the definition
of a 2-group, so $(\G^{(1)},\circ)$ forms a group. 

Moreover, there are 2-morphisms between the 1-morphisms. They have
both a horizontal and a vertical composition. The horizontal
composition we denote $\circ$, since this is the same direction as the
composition of 1-morphisms. The vertical composition we denote
$\cdot$. The two compositions must be compatible, in the sense that
the following composite is well-defined:
\begin{equation}
  \xymatrix{
    \star & & \star \ar@/^2pc/[ll]^{\gamma_1}="2" \ar[ll]|{\gamma_2}="1" \ar@/_2pc/[ll]_{\gamma_3}="0" & & \star \ar@/^2pc/[ll]^{\gamma'_1}="5" \ar[ll]|{\gamma'_2}="4" \ar@/_2pc/[ll]_{\gamma'_3}="3" \\
    \ar@{=>}"0"+<0ex,-2ex> ;"1"+<0ex,+2ex>^{\chi_2} \\
    \ar@{=>}"1"+<0ex,-2ex> ;"2"+<0ex,+2ex>^{\chi_1} \\
    \ar@{=>}"3"+<0ex,-2ex> ;"4"+<0ex,+2ex>^{\chi'_2} \\
    \ar@{=>}"4"+<0ex,-2ex> ;"5"+<0ex,+2ex>^{\chi'_1}
  }
\end{equation}

This is expressed as the ``interchange law''
\begin{equation}\label{eq:interchange}
  (\chi_1 \cdot \chi_2) \circ (\chi'_1 \cdot \chi'_2) = (\chi_1 \circ \chi'_1) \cdot (\chi_2 \circ \chi'_2)
\end{equation}

Again, these 2-morphisms are assumed to be invertible, so it follows
that $(\G^{(2)},\circ)$ forms a group. Note that we do not say that
$(\G^{(2)},\cdot)$ forms a group, since $\cdot$ may not be defined for
all pairs of 2-morphisms, unless they have compatible source and
target 1-morphisms.

However, one can find a group structure using vertical composition of
2-morphisms, upon choosing a given 1-morphism $\gamma$, and
considering $Hom(\gamma,\gamma)$, the collection of 2-morphisms from
$\gamma$ to itself.  This is a group with operation $\cdot$, since
2-morphisms are invertible under $\cdot$.

(Note that $Hom(\gamma,\gamma)$ does not necessarily close under $\circ$,
unless $\gamma=1_G$. In that case, $Hom(1_G,1_G)$ is an abelian group by the
so-called Eckmann-Hilton argument, which uses the interchange property
(\ref{eq:interchange}) to show that horizontal and vertical
composition agree, and define an abelian group.)

We have noted the composition operations for both types of
morphisms. It is also significant that there is also a straightforward
interaction between the two types, namely whiskering. That is, a
1-morphism can act on a 2-morphism, either on the left or the right,
where it acts by composition with an identity 2-cell. That is, by
convention we say:
\begin{equation}
  \qquad{
    \xymatrix{
      \star & & \star \ar[ll]^{\gamma} & & \star \ar@/^2pc/[ll]^{\gamma_2}="1" \ar@/_2pc/[ll]_{\gamma_1}="0" \\
      \ar@{=>}"0"+<0ex,-2ex> ;"1"+<0ex,+2ex>^{\chi}
    }
  } = \qquad{
    \xymatrix{
      \star & & \star \ar@/^2pc/[ll]^{\gamma \circ \gamma_1}="3"  \ar@/_2pc/[ll]_{\gamma \circ \gamma_2}="2" \\
      \ar@{=>}"2"+<0ex,-2ex> ;"3"+<0ex,+2ex>^{Id_{\gamma} \circ \chi}
    }
  }
\end{equation}
and similarly for whiskering on the right.

Again, all these remarks are only unpacking what is implied by the
definition of a 2-category with one object and invertible 1- and 
2-morphisms. We will want to use this definition occasionally,
particularly when describing 2-group actions. However, for the most
part is will be useful to think of 2-groups in a way which is more
amenable to explicit calculations. The first step toward this is the
definition of a categorical group, and then even more explicit is the
presentation in terms of a crossed module. We recall these next.

\subsubsection{Categorical Groups}

Another natural approach to defining a ``higher'' analog of a group is
based on the view that a group is a ``group object in $\cat{Set}$''. A group object in a category $\cat{C}$ with finite products 
(in particular, $\cat{C}=\cat{Set}$ or $\cat{C}=\cat{Cat}$ with Cartesian products)
is an object $G \in
\cat{C}$, equipped with structure maps such as the multiplication map $m: G \times G
\rightarrow G$ satisfying the same axioms as those for a group. These
axioms can be expressed as commuting diagrams which make sense in any
category with finite products and specialize to the usual axioms in
$(\cat{Set},\times)$. 

\begin{definition}A (strict) \defn{categorical group} is a (strict) group
  object in the monoidal category $(\cat{Cat},\times)$.
\end{definition}

We remark here that there is a more general concept of (weak)
categorical groups, but we will restrict our attention to the strict case
in the present paper. See Baez and Lauda \cite{hda5} for an exposition of the weak case.

That is, categorical groups are (small) categories $\cat{G}$, equipped with
functors $\otimes : \cat{G} \times \cat{G} \rightarrow \cat{G}$, and
$inv : \cat{G} \rightarrow \cat{G}$ satisfying the group axioms.

Since $\cat{Cat}$ is, in fact, a monoidal 2-category (in which the
monoidal product on objects is the cartesian product of categories), our
definition specifies that the functors $\otimes$ and $inv$  satisfy the 
axioms for a group \textit{strictly}. 
That is, the usual equations such as associativity
of multiplication are still equations, rather than 2-isomorphisms. The
corresponding weak notion, in which equations are replaced by such
2-isomorphisms, necessarily has additional coherence conditions they
must satisfy.

Notice that the ``group multiplication'' functor, which we have
written as $\otimes$, satisfies the properties for a monoidal product,
as the notation suggests. Indeed, a (strict) monoidal category is
simply a monoid object in $\cat{Cat}$. A (strict) categorical group is
therefore a strict monoidal category with inverses. That is, every
object and morphism has an inverse with respect to $\otimes$.

It is standard, and easy to see, that any 2-group as defined in
Section \ref{sec:2group-2cat} gives rise to a categorical group, and
vice versa. We describe the correspondence here to settle notation.

\begin{definition}
  Given a 2-group $\G$, the categorical group $(\cat{C}(\G),\otimes)$
  associated to $\G$ is defined as follows:
  \begin{itemize}
  \item \textbf{Objects}: $\cat{C}(\G)^{(0)} = \G^{(1)}$ (objects are
    morphisms of $\G$)
  \item \textbf{Morphisms}: $\cat{C}(\G)^{(1)} = \G^{(2)}$ (morphisms
    are 2-morphisms of $\G$)
    \item The composition of morphisms of $\cat{C}(\G)$ is vertical
      composition of 2-morphisms in $\G$ 
    \item The monoidal product $\otimes$ of $\cat{C}(\G)$ is:
      \begin{itemize}
      \item On $\cat{C}(\G)^{(0)}$, same as composition in $\G^{(1)}$
      \item On $\cat{C}(\G)^{(1)}$, same as horizontal composition in
        $\G^{(2)}$
      \end{itemize}
    \item The inverse functor $inv : \cat{C}(\G) \rightarrow
      \cat{C}(\G)$ is given, for each object or morphism, by the
      inverse or vertical inverse of the corresponding 1- or 2- morphism in $\G$
  \end{itemize}

  Given a categorical group $(\cat{G},\otimes)$, the 2-group $\star
  \wquot \cat{G}$ has:
  \begin{itemize}
    \item \textbf{Object}: just one, $\star$
    \item \textbf{Morphisms}: $(\star \wquot \cat{G})^{(1)} =
      \cat{G}^{(0)}$, the objects of $\cat{G}$
    \item \textbf{2-Morphisms}: $(\star \wquot \cat{G})^{(2)} =
      \cat{G}^{(1)}$, the morphisms of $\cat{G}$
    \item Composition for $(\star \wquot \cat{G})^{(1)}$ and
      horizontal composition for $(\star \wquot \cat{G})^{(2)}$ are
      $\otimes^{(0)}$ and $\otimes^{(1)}$ from $\cat{G}$ respectively
    \item Vertical composition for $(\star \wquot \cat{G})^{(2)}$ is
      composition for $\cat{G}^{(1)}$
  \end{itemize}
\end{definition}

It is well known, and the unfamiliar reader may easily check, that
these two correspondences give an equivalence of the two
definitions. For example, the properties of invertibility for $\star
\wquot \cat{G}$ follow from the existence of $inv$ in
$(\cat{G},\otimes)$, and the fact that it satisfies group axioms. It
is similarly easy to check that monoidal functors between categorical
groups correspond to 2-functors between 2-groups. For example, the
interchange law for 2-groups is expressed, in categorical groups, as
the fact that the monoidal product is functorial. The convention for whiskering in a 2-group is easily translated to
the usual convention for the monoidal product of an object with a
morphism.

Another equivalent presentation is a result of the fact that there is
a correspondence between groups internal to categories, and categories
internal to groups (that is, categories whose sets of objects and sets
of morphisms each have the structure of a group, and where source,
target, and composition operations are all group homomorphisms). This
equivalence follows naturally, since given a categorical group, the
structure maps, such as $\otimes$, are functors which satisfy the
group axioms. The object and morphism maps for these functors
therefore each separately satisfy the same axioms. Thus, the objects
of a categorical group form a group, as do the morphisms of the
categorical group. Thinking of these separate group structures
naturally leads to the third of the definitions we will use in this
paper, namely crossed modules.

\subsubsection{Crossed Modules}

A well-known theorem due to Brown and Spencer \cite{brownspencer2} 
says that (strict) 2-groups
are classified by \defn{crossed modules}. See also \cite{brown-higgins-sivera} for further background.

\begin{definition}
  A crossed module consists of $(G,H,\rhd,\partial)$, where $G$ and
  $H$ are groups, $G \rhd H$ is an action of $G$ on $H$ by automorphisms
  and $\partial : H \ra G$ is a
  homomorphism, satisfying the equations:
  \begin{equation}\label{cm1}
    \partial(g \rhd \eta) = g \partial(\eta) g^{-1}
  \end{equation}
  and
  \begin{equation}\label{cm2}
    \partial(\eta) \rhd \zeta = \eta \zeta \eta^{-1}
  \end{equation}
\end{definition}

We describe the correspondence with categorical groups, which will be
the form we make the most use of, though of course the correspondence
with 2-groups follows immediately as well:

\begin{definition}
  The categorical group $(\cat{G},\otimes)$ given by $(G,H,\rhd,\partial)$ has:
  \begin{itemize}
  \item \defn{Objects}: $\cat{G}^{(0)} = G$
  \item \defn{Morphisms}: $\cat{G}^{(1)} = G \times H$, with source
    and target maps
    \begin{equation}\label{cm3}
      s(g,\eta) = g
    \end{equation}
    and
    \begin{equation}\label{cm4}
      t(g,\eta)  = \partial(\eta)g
    \end{equation}
  \item \defn{Identities}:
    \begin{equation}
      Id_g = (g,1_H)
    \end{equation}
  \item \defn{Composition}:
    \begin{equation}\label{cm5}
      (\partial(\eta)g,\zeta) \circ (g, \eta) = (g, \zeta \eta).
    \end{equation}
  \item \defn{Monoidal product}: given on objects by $g_1\otimes g_2=g_1g_2$ and on morphisms by 
    \begin{equation}\label{cm6}
    (g_1,\eta)\otimes(g_2,\zeta) = (g_1 g_2, \eta (g_1 \rhd \zeta) )
    \end{equation}
  which corresponds to the horizontal composition of 2-morphisms in $\G$.
  \end{itemize} 
  \label{defCG}
\end{definition}

The theorem alluded to above asserts that any strict 2-group is
equivalent to one of this form. We will usually assume that any
2-group we use from now on is presented by a crossed module, and use
the corresponding notation for explicit calculations.

The structure of a category which lies behind the definition of a
crossed module is somewhat disguised. Much of it is included in the
fact that $G$ and $H$ are groups. The explicit action of $G$ on $H$
captures the notion of whiskering, remarked upon above, and the 
projection onto the first factor, together
with the ``boundary'' map $\partial$, determines the source and target
structure maps for the categorical group. However, for example, the
vertical composition of 2-morphisms is not explicitly mentioned. Full
details of how it and other categorical group structures can be
reconstructed from a crossed module are well recorded elsewhere.

\subsection{Double Groupoid of a 2-Group}\label{sec:doublegroupoid}

A useful construction for calculations later on is the double groupoid
of $\G$.

Double categories were introduced by Ehresmann
\cite{ehresmann,quintets}, and can be defined in several equivalent
ways, three of which are noted by Brown and Spencer
\cite{brownspencer}. The most pertinent later in this paper is that
they are categories \textit{internal} to $\cat{Cat}$, the category of
all categories. This terse definition is the most common in current
use, but is somewhat opaque, and obscures the basically symmetric
nature of these structures.

One can also describe a (small) double category in a more manifestly
symmetric way. Although the two definitions are equivalent, they are
superficially rather different. In Section \ref{sec:transdouble} we
will use both points of view to describe the structure which it is the
goal of this paper to construct. Therefore, to avoid confusion, we
will reserve ``double category'' for the symmetric view, and use
``$\cat{Cat}$-category'' to mean the equivalent structure seen as a
category in $\cat{Cat}$.

We will not give a full definition of ``double category'' here,
though, since it is equivalent to and can be deduced from the terse
definition, but it is intuitively useful to see a double category
${\cal D}$ as consisting of:
\begin{itemize}
\item a set of objects $O$
\item two sets $\cal H$ and $\cal V$ of morphisms (we use calligraphic letters here to avoid confusion with the group $H$), denoted \textit{horizontal}
  and \textit{vertical}, together with structure maps making them into
  horizontal and vertical categories with objects $O$
\item a set of squares $S$ which form the morphisms of two category
  structures whose objects are $\cal H$ and $\cal V$ respectively
\end{itemize}
The category structures on the squares $S$ satisfy some extra
properties making them compatible with the category structures on $\cal H$ 
and $\cal V$.  Intuitively, they say that ``horizontal'' structure maps
commute with ``vertical'' structure maps. For example, one can take a
horizontal composite of two squares $\sigma_1$ and $\sigma_2$, and
then take the morphism $s_v(\sigma_1 \circ_h \sigma_2) \in \cal{H}$ which is
its vertical source. This is the same as $s_v(\sigma_1) \circ_h
s_v(\sigma_2)$, the horizontal composite of the vertical sources of
each square. There are many other conditions of this form, such as the
``interchange law'' for composition of squares, which has the same
form as (\ref{eq:interchange}).

A double groupoid is a double category in which all morphisms and
squares are invertible. Further discussion of double groupoids is
found in Brown and Spencer \cite{brownspencer}, while an interesting
summary of general folklore about this symmetric definition and
variations on it can be found online \cite{nlab-doublecat}.

The above intuition essentially justifies a graphical notation which
makes 2-group calculations more straightforward.  This notation has been
extensively used elsewhere, see e.g. \cite{brown-higgins-sivera},
and we now describe it for reference. 

\begin{definition} Given a 2-group $\G$, the double groupoid of $\G$,
  denoted ${\cal D}(\G)$, is the double groupoid with:
\begin{itemize}
\item A unique object of ${\cal D}(\G)$: $O = \G^{(0)} = \{ \star \}$
\item Horizontal and vertical morphisms of ${\cal D}(\G)$ are $\cal{H} = \cal{V} =
  \G^{(1)}$
\item Squares of ${\cal D}(\G)$ are given by all squares of
  the form
\begin{equation*}
\xybiglabels \vcenter{\xymatrix@M=0pt@=3pc{\ar@{-} [d] _{g_4} \ar@{-} [r]^{g_3} \ar@{} [dr]|\eta & \ar@{-} [d]^{g_2} \\
\ar@{-} [r]_{g_1}  & }}
\end{equation*}
where $g_i\in G, \, i=1, \ldots, 4$, $\eta\in H$ and $\partial(\eta)= g_1g_2g_3^{-1}g_4^{-1}$
\item horizontal and vertical composition of squares is given by
\begin{equation*}
\xybiglabels \vcenter{\xymatrix @=3pc @W=0pc @M=0pc { \ar@{-}[r] ^{g_3} \ar@{-}[d]
_{g_4} \ar@{}[dr]|{\eta_1} & \ar@{-}[r] ^{g_7} \ar@{-}[d]|{g_2}
\ar@{}[dr]|{\eta_2} &  \ar@{-}[d]^{g_6} 
\\ \ar@{-}[r] _{g_1} & \ar@{-}[r] _{g_5} & 
}}
 \,  = \,
\xybiglabels \vcenter{\xymatrix@M=0pt@=3pc@C=3pc{\ar@{-} [d] _{g_4} \ar@{-} [rrr]^-{g_3g_7} & \ar@{} [dr]|-{\eta_1 (g_4g_3g_2^{-1})\rhd \eta_2}& &\ar@{-} [d]^{g_6} \\
\ar@{-} [rrr]_-{g_1g_5}  & &&}}
 \,  = \,
\xybiglabels \vcenter{\xymatrix@M=0pt@=3pc@C=2pc{\ar@{-} [d] _{g_4} \ar@{-} [rrr]^-{g_3g_7} & \ar@{} [dr]|-{(g_1\rhd \eta_2)\eta_1 }& &\ar@{-} [d]^{g_6} \\
\ar@{-} [rrr]_-{g_1g_5}  & &&}}
\end{equation*}
where the two expressions for the $H$ element in the horizontal compositions are the same, using $\partial(\eta_1^{-1})\rhd \eta_2=\eta_1^{-1} \eta_2\eta_1$, and
\begin{equation*}
\xybiglabels \vcenter{\xymatrix@M=0pt@=3pc{\ar@{-} [d] _{g_4} \ar@{-} [r]^{g_3} \ar@{} [dr]|{\eta_1} & \ar@{-} [d]^{g_2} \\
\ar@{-} [r] |{g_1} \ar@{-} [d]_{g_5} \ar@{}[dr] |{\eta_2} & \ar@{-} [d]^{g_7} \\
\ar@{-} [r]_{g_6}& }}  \,  = \,
\xybiglabels \vcenter{\xymatrix@M=0pt@=3pc@C=2pc{\ar@{-} [d] _{g_5g_4} \ar@{-} [rrr]^-{g_3} & \ar@{} [dr]|-{\eta_2(g_5\rhd \eta_1)}& &\ar@{-} [d]^{g_7g_2} \\
\ar@{-} [rrr]_-{g_6}  & &&}}
\end{equation*}
and these operations satisfy the interchange law, i.e. the equality of
evaluating a 2 by 2 array of squares in two different ways (first
horizontal and then vertical composition, or vice-versa).
\end{itemize}
\end{definition}

In Ehresmann's terminology, this is the double category of
\textit{quintets} of the bicategory $\G$: the term refers to the fact
that squares are determined by five pieces of data: the four
1-morphisms which are its edges, and the 2-morphism which fills the
square.

\begin{remark}
  Due to the interchange law there is a consistent evaluation of any
  rectangular array of squares, which is independent of the order in
  which the horizontal and vertical multiplications are
  performed. Also squares in ${\cal D}(\G)$ have horizontal and
  vertical inverses, which are respectively given by (for the square
  in the definition):
\begin{equation*}
\xybiglabels \vcenter{\xymatrix@M=0pt@=3pc@C=1pc{\ar@{-} [d] _{g_2} \ar@{-} [rrr]^-{g_3^{-1}} & \ar@{} [dr]|-{\eta^{-h}}& &\ar@{-} [d]^{g_4} \\
\ar@{-} [rrr]_-{g_1^{-1}}  & &&}} \, = \, 
\xybiglabels \vcenter{\xymatrix@M=0pt@=3pc@C=2pc{\ar@{-} [d] _{g_2} \ar@{-} [rrr]^-{g_3^{-1}} & \ar@{} [dr]|-{g_1^{-1}\rhd \eta^{-1}}& &\ar@{-} [d]^{g_4} \\
\ar@{-} [rrr]_-{g_1^{-1}}  & &&}}
\qquad
\xybiglabels \vcenter{\xymatrix@M=0pt@=3pc@C=1pc{\ar@{-} [d] _{g_4^{-1}} \ar@{-} [rrr]^-{g_1} & \ar@{} [dr]|-{\eta^{-v}}& &\ar@{-} [d]^{g_2^{-1}} \\
\ar@{-} [rrr]_-{g_3}  & &&}}  \, = \,
\xybiglabels \vcenter{\xymatrix@M=0pt@=3pc@C=2pc{\ar@{-} [d] _{g_4^{-1}} \ar@{-} [rrr]^-{g_1} & \ar@{} [dr]|-{g_4^{-1}\rhd \eta^{-1}}& &\ar@{-} [d]^{g_2^{-1}} \\
\ar@{-} [rrr]_-{g_3}  & &&}}
\end{equation*}

\end{remark}

In particular, the double groupoid ${\cal D}(\G)$ contains a copy of
$\G$ by considering only the horizontal morphisms, and the squares for
which the vertical source and target are identities. Thus, a typical
square of this sort is of the following form:
\begin{equation*}
\xybiglabels \vcenter{\xymatrix@M=0pt@=3pc{\ar@{-} [d] \ar@{-} [r]^{g_1} \ar@{} [dr]|\eta & \ar@{-} [d] \\
\ar@{-} [r]_{g_2}  & }}
\end{equation*}
where here, and henceforth, any unlabelled edge or square is taken to be labelled with the identity of the corresponding group. 
This can be identified with a 2-morphism in the 2-group $\G$, namely:
\begin{equation}
  \xymatrix{
    \star & & \star  \ar@/^2pc/[ll]^{g_2}="1" \ar@/_2pc/[ll]_{g_1}="0" \\
    \ar@{=>}"0"+<0ex,-2ex> ;"1"+<0ex,+2ex>^{\eta} \\
  }
\end{equation}
or equivalently, with the morphism $(g_1,\eta)$ in the categorical group.

\section{Actions of 2-Groups on Categories}\label{sec:2group-actions}

Our aim is to describe actions of a 2-group $\G$ on a category
$\C$. There are two equivalent ways of describing these, depending on
whether one views $\G$ as a 2-group or a categorical group. In this
section, we will outline these two views and see the relation between
them. Then we will consider a natural example, namely the adjoint
action of a categorical group on its own underlying category.

There are two different, but equivalent, definitions of group actions
on a set, and both will be relevant for us.  To begin with, consider a
group $G$, seen as a one-object category whose morphisms are all
invertible.  Then a $G$ action $\phi$ on a set $X$ may be described as
a functor
\begin{equation}
\phi : G \ra \cat{Set}
\label{eq:G-action-functor}
\end{equation}
where $X = \phi(\star)$ is the image of the unique object of $G$,
i.e. the image of $\phi$ is $End(X)$, the full subcategory of
$\cat{Set}$ with the single object $X$ (in fact, since $G$ is a
group, the image is $Aut(X)$, consisting of only the invertible
endomorphisms of $X$).

Next we show how to construct the transformation groupoid in a way
that naturally generalizes to the 2-group case, by regarding the group
$G$ as a set equipped with a multiplication map $m : G \times G \ra G$
and an inverse $inv : G \ra G$, satisfying the group axioms. Thus we
have the familiar definition of an action on a set $X$ as a function
\begin{equation}
  \hat{\phi} : G \times X \ra X
\label{eq:G-action-function}
\end{equation}
This function is related to $\phi$ by
$\hat{\phi}(\gamma,x)=\phi_\gamma(x)$.  Functoriality of $\phi$ means,
in particular,
that $\hat{\phi}$ satisfies a compatibility condition with the
multiplication map $m : G \times G \ra G$, namely that the
following commutes:
\begin{equation}\label{eq:actioncondition-diag}
  \xymatrix@C=+5pc{
    G \times G \times X \ar[r]^{m \times Id_X} \ar[d]_{Id_G \times \hat{\phi}} & G \times X \ar[d]^{\hat{\phi}} \\
    G \times X \ar[r]_{\hat{\phi}} & X
  }
\end{equation}
$\hat{\phi}$ also satisfies a unit condition:
\begin{equation}
  \hat{\phi}(1,x) =x, \, \forall x\in X
\label{eq:phi-unit-condition}
\end{equation}

This definition is, of course, equivalent to the point of view of an
action as a functor. First it determines $\phi(\star) = X$. The two
definitions are then related by turning a function $ G \ra Hom(X,X)$
into an $X$-valued function of $G \times X$, taking $(\gamma,x)$ to
$\phi_\gamma(x)$. (The term for this in logic is ``uncurrying'', while
``currying'' denotes the process which turns a function of $n$
variables into a chain of $n$ one-variable functions, each one
returning the next function in the chain).

\subsection{Actions of 2-Groups on Categories}

Next we extend the two viewpoints of the previous introduction for the
action of a group on a set to the action of a 2-group on a
category. First, by analogy with (\ref{eq:G-action-functor}),
regarding a 2-group as a 2-category, an action will be a 2-functor
into the 2-category $\cat{Cat}$. 

\begin{definition}\label{def:2grp_action}
  A 2-group $\G$ acts (strictly) on a category $\C$ if there is a
  (strict) 2-functor:
\begin{equation}
  \Phi : \G \ra \cat{Cat}
\end{equation}
whose image lies in $End(\C)$, the full sub-2-category of $\cat{Cat}$
with the single object $\C$.
\end{definition}

Thus $\Phi(\ast)=\C$, and on 1- and 2-morphisms $\Phi$ is given by assignments:
\begin{itemize}
    \item for each $\gamma \in \G^{(1)}$ we have the endofunctor 
	$\Phi_\gamma: \C\rightarrow \C$, acting as 
	$$(x\stackrel{f}{\rightarrow}y) \mapsto 
	(\Phi_\gamma(x)\stackrel{\Phi_\gamma(f)}{\longrightarrow}\Phi_\gamma(y))$$
      \item for each 2-morphism $(\gamma_1, \chi) \in \G^{(2)}$, we
        have the natural transformation $\Phi_{(\gamma_1,\chi)}:
        \Phi_{\gamma_1}\rightarrow \Phi_{\gamma_2}$, where
        $\gamma_2= \partial(\chi)\gamma_1$, given by assignments
        $(\C^{(0)} \ni x) \mapsto (\Phi_{(\gamma_1,\chi)}(x) \in
        \C^{(1)})$ satisfying the naturality condition
\begin{equation}
\vcenter{\xymatrix@=3.5pc{
    \Phi_{\gamma_1}(x)   \ar[r]^{\Phi_{\gamma_1}(f)} \ar [d]_{\Phi_{(\gamma_1,\chi)}(x)}  \ar@{} [dr]|{} & \Phi_{\gamma_1}(y)  \ar[d]^{\Phi_{(\gamma_1,\chi)}(y)}\\
    \Phi_{\gamma_2}(x) \ar[r]_{\Phi_{\gamma_2}(f)} &   \Phi_{\gamma_2}(y)}} 
\label{eq:Phinatural}
\end{equation} 
\end{itemize}
These assignments must  satisfy the conditions to be a strict 2-functor, i.e. they preserve all vertical and horizontal compositions and identities, which here means:

\begin{itemize}
    \item[Ver] 
	\begin{itemize}
    		\item[1)] $\Phi_{(\gamma_2,\chi_2)}(x) \circ_v \Phi_{(\gamma_1,\chi_1)}(x) = \Phi_{(\gamma_1,\chi_2\chi_1)}(x)$ where $\gamma_2= \partial(\chi_1)\gamma_1$
    		\item[2)] $\Phi_{(\gamma, 1)}(x) = {\rm id}_{\Phi_\gamma(x)}$
	\end{itemize}
    \item[Hor] 
	\begin{itemize}
    		\item[1)] $\Phi_{\gamma_1}\circ \Phi_{\gamma_3} = \Phi_{\gamma_1\gamma_3}, \,  \Phi_1 = {\rm id}_{\C}  $
    		\item[2)] $\Phi_{(\gamma_1,\chi_1)} \circ_h \Phi_{(\gamma_3,\chi_2)} = \Phi_{(\gamma_1\gamma_3,\chi_1(\gamma_1 \rhd\chi_2))}$. 

				Writing out explicitly the horizontal composition of natural transformations on the l.h.s., this final condition becomes: 
				$$
				\Phi_{(\gamma_1,\chi_1)}(\Phi_{\gamma_4}(x)) \circ \Phi_{\gamma_1}(\Phi_{(\gamma_3,\chi_2)} (x) )=
				\Phi_{(\gamma_1\gamma_3,\chi_1(\gamma_1 \rhd\chi_2))}(x)
				$$
				where $\gamma_4= \partial(\chi_2)\gamma_3$. Here the underlying array of squares in ${\cal D}(\G)$ is 
				\begin{equation}
				\xybiglabels \vcenter{\xymatrix @=3pc @W=0pc @M=0pc { \ar@{-}[r] ^{\gamma_1} \ar@{-}[d]_{} \ar@{}[dr]|{\chi_1} & \ar@{-}[r] ^{\gamma_3} \ar@{-}[d]|{}
				\ar@{}[dr]|{\chi_2} & \ar@{-}[d]^{}
				\\ \ar@{-}[r] _{\gamma_2} & \ar@{-}[r] _{\gamma_4} & 
				}}
				\label{eq:F2squares}
				\end{equation}

	\end{itemize}
\end{itemize}

The second viewpoint for a 2-group to act on a category is to regard the 
2-group $\G$ as a categorical group, i.e. a monoidal category, and proceed by analogy with
(\ref{eq:G-action-function}), (\ref{eq:actioncondition-diag}).

\begin{definition}
  A strict action of a categorical group $\G$ on a category $\C$ is a
  functor $\hat{\Phi}:\G\times \C \rightarrow \C$ 
  satisfying the action square diagram 
  in $\cat{Cat}$ (strictly):
\begin{equation}\label{eq:catactioncondition-diag}
  \xymatrix@C=+5pc{
    \G \times \G \times \C \ar[r]^{\otimes \times Id_{\C}} \ar[d]_{Id_{\G} \times \hat{\Phi}} & \G \times \C \ar[d]^{\hat{\Phi}} \\
    \G \times \C \ar[r]_{\hat{\Phi}} & \C
  }
\end{equation}
and the unit condition 
\begin{equation}\label{eq:Phi-unit-condition}
 \hat{\Phi}(1,x) =x, \, \forall x\in \C^{(0)} 
\end{equation}
\label{def:2grp_action-Phi-hat}
\end{definition}

Thus, from functoriality, we have the conditions:
\begin{eqnarray}
\hat{\Phi}((\gamma_2,\chi_2), g) \circ  \hat{\Phi}((\gamma_1,\chi_1), f) 
& = & \hat{\Phi}((\gamma_1,\chi_2\chi_1), g\circ f)
\label{eq:Phihat-func1}\\
\hat{\Phi}((\gamma,1_H), {\rm id}_x) & = & {\rm id}_{\hat{\Phi}(\gamma,x)}
\label{eq:Phihat-func2}
\end{eqnarray}
where, in (\ref{eq:Phihat-func1}),  $\gamma_2= \partial(\chi_1)\gamma_1$. 
The action square diagram corresponds to the equations (on objects and morphisms
respectively):
\begin{equation}
\hat{\Phi}(\gamma_1\gamma_3, x) = \hat{\Phi}(\gamma_1, \hat{\Phi}(\gamma_3,x))
\label{eq:Phihat-ASobjects}
\end{equation}
\begin{equation}
\hat{\Phi}((\gamma_1\gamma_3,\chi_1(\gamma_1\rhd \chi_2) ), f) =  \hat{\Phi}((\gamma_1,\chi_1), \hat{\Phi}((\gamma_3,\chi_2), f)).
\label{eq:Phihat-ASmorphisms}
\end{equation}
Here the underlying array of squares in ${\cal D}(\G)$ is (\ref{eq:F2squares}).

The following theorem shows that these two viewpoints are equivalent.
\begin{theorem}\label{thm:phifunctorial}
  A strict 2-functor $\Phi : \G \ra End(\C)$ is equivalent to a strict
  action functor $\hat{\Phi}:\G\times \C \rightarrow \C$.
\end{theorem}
\begin{proof}
Given $\Phi$, set:
$ \hat{\Phi}(\gamma,x) := {\Phi}_\gamma(x)$ and 
\begin{equation}
  (\xybiglabels{\xymatrix{\hat{\Phi}(\gamma_1,x) \ar[rr]^-{\hat{\Phi}((\gamma_1,\chi), f)} && \hat{\Phi}(\gamma_2,y)   }})
  :=
  (\xybiglabels{\xymatrix{{\Phi}_{\gamma_1}(x) \ar[rrr]^-{\Phi_{(\gamma_1,\chi)}(y)\circ \Phi_{\gamma_1}(f)} &&& {\Phi}_{\gamma_2}(y)   }})
\label{eq:PhihatfromPhi}
\end{equation}
By the naturality condition (\ref{eq:Phinatural}), this is also equal
to
\begin{equation}
(\xybiglabels{\xymatrix{{\Phi}_{\gamma_1}(x) \ar[rrr]^-{\Phi_{\gamma_2}(f)  \circ \Phi_{(\gamma_1,\chi)}(x)} &&& {\Phi}_{\gamma_2}(y)   }})
\end{equation}
Now functoriality for $\hat{\Phi}$ follows from putting together four
naturality squares (\ref{eq:Phinatural}) in the obvious way, and then
using functoriality of $\Phi_\gamma$ horizontally and the first Ver
condition vertically.

The first action square equation (\ref{eq:Phihat-ASobjects})  and the unit condition (\ref{def:2grp_action-Phi-hat}) are immediate. The second action square equation (\ref{eq:Phihat-ASmorphisms}) follows from:
\begin{eqnarray*}
\hat{\Phi}((\gamma_1\gamma_3,\chi_1(\gamma_1\rhd \chi_2) ), f)  & = & \Phi_{(\gamma_1\gamma_3,\chi_1(\gamma_1\rhd \chi_2))}(y)\circ \Phi_{\gamma_1\gamma_3}(f) \\
& = & \Phi_{(\gamma_1,\chi_1)}(\Phi_{\gamma_4}(y)) \circ (\Phi_{\gamma_1}(\Phi_{(\gamma_3,\chi_2)}(y)) \circ \Phi_{\gamma_1}(\Phi_{\gamma_3}(f)) ) \\
& = & \Phi_{(\gamma_1,\chi_1)}(\Phi_{\gamma_4}(y)) \circ (\Phi_{\gamma_1}(\Phi_{(\gamma_3,\chi_2)}(y)) \circ \Phi_{\gamma_3}(f)) ) \\
& = & \hat{\Phi}((\gamma_1, \chi_1),  \Phi_{(\gamma_3,\chi_2)}(y) \circ \Phi_{\gamma_3}(f)   ) \\
& = & \hat{\Phi}((\gamma_1,\chi_1), \hat{\Phi}((\gamma_3,\chi_2), f))
\end{eqnarray*}
using the first Hor condition for $\Phi$ and associativity in the second equality.

Conversely, given $\hat{\Phi}$, set 
${\Phi}_\gamma(x)  := \hat{\Phi}(\gamma,x) $ (on objects), 
$$
(\xybiglabels{\xymatrix{{\Phi}_{\gamma}(x) \ar[r]^-{ \Phi_{\gamma}(f)} & 
{\Phi}_{\gamma}(y)  }})
:=
(\xybiglabels{\xymatrix{\hat{\Phi}(\gamma,x) \ar[rr]^-{\hat{\Phi}((\gamma,1_H), f)} && 
{\Phi}(\gamma,y)  }})
$$ 
(on morphisms), and on 2-morphisms the natural transformation $\Phi_{(\gamma_1, \chi)}: \Phi_{\gamma_1}\rightarrow \Phi_{\gamma_2}$, where $\gamma_2= \partial(\chi)\gamma_1$, is given by:
\begin{equation}
  (\xybiglabels{\xymatrix{{\Phi}_{\gamma_1}(x) \ar[rr]^-{\Phi_{(\gamma_1,\chi)}(x)} && 
      {\Phi}_{\gamma_2}(x) }})
  :=
  (\xybiglabels{\xymatrix{\hat{\Phi}(\gamma_1,x) \ar[rr]^-{\hat{\Phi}((\gamma_1,\chi), {\rm id}_x)} &&  \hat{\Phi}(\gamma_2,x) }})
\end{equation}
Now, the Hor properties for $\Phi$ follow in a straightforward manner using the functoriality conditions (\ref{eq:Phihat-func1}), (\ref{eq:Phihat-func2}) for $\hat{\Phi}$.
The first Ver property for $\Phi$ is immediate, and the second Ver property follows from:
\begin{eqnarray*}
\Phi_{(\gamma_1,\chi_1)}(\Phi_{\gamma_4}(x)) \circ \Phi_{\gamma_1}(\Phi_{(\gamma_3,\chi_2)} (x) ) 
& = & \hat{\Phi}((\gamma_1,\chi_1), {\rm id}_{\hat{\Phi}(\gamma_4,x)}) \circ \hat{\Phi}((\gamma_1,1_H) , \hat{\Phi}((\gamma_3,\chi_2), {\rm id}_x)) \\
& = & \hat{\Phi}((\gamma_1,\chi_1),  \hat{\Phi}((\gamma_3,\chi_2), {\rm id}_x)) \\
& = & \hat{\Phi}((\gamma_1\gamma_3,\chi_1(\gamma_1 \rhd\chi_2)), {\rm id}_x)) \\
 & = & \Phi_{(\gamma_1\gamma_3,\chi_1(\gamma_1 \rhd\chi_2))}(x)
\end{eqnarray*}
where we use functoriality of $\hat{\Phi}$ (\ref{eq:Phihat-func1}) in the second equality and the second action square condition
 (\ref{eq:Phihat-ASmorphisms}) applied to  $f={\rm id}_x$ in the penultimate equality.
\end{proof}

It is convenient to introduce a succinct notation for a 2-group action
analogous to the usual notation $g \rhd x = \phi_g(x)$ for a group
action. There are actually three possible maps.

Since $\Phi_{\gamma}$ is a functor with both object and morphism maps,
the objects of $\G$ act on both the set of objects and the set of
morphisms of $\C$. Moreover, the action functor $\hat{\Phi}$ also has
both object and morphism maps. The object map 
$\hat{\Phi}(\gamma,-)$ is the same as the object map for
$\Phi_{\gamma}$, by the above argument. However, the morphism map
determines an action of the morphisms of $\G$ on the morphisms of
$\C$, which is a different action again.

We will return to the relation between these three actions again in
Corollary \ref{cor:transgpds-3actions}. For the moment, is convenient
to use the symbol $\act$ to denote all three, as follows.

\begin{definition}\label{def:act-notation}
  If $\G$ is a 2-group classified by the crossed module
  $(G,H,\rhd,\partial)$, let the notation $\act$ denote the
  following.
  \begin{itemize}
  \item Given $\gamma \in \G^{(0)} = G$ and $x \in \C^{(0)}$, let
    \begin{equation}
      \gamma \act x = \Phi_\gamma(x) = \hat{\Phi}(\gamma, x) 
    \end{equation}
  \item Given $\gamma \in \G^{(0)} = G$ and $f \in \C^{(1)}$, let
    \begin{equation}
      \gamma \act f = \Phi_\gamma(f) = \hat{\Phi}((\gamma,1_H), f) 
    \end{equation}
  \item Given $(\gamma,\chi) \in \G^{(1)} = G \ltimes H$ and $(f: x
    \ra y) \in \C^{(1)}$, let
    \begin{equation}
      \begin{array}{ccl}
        (\gamma, \chi)\act f & = & \hat{\Phi} ((\gamma, \chi), f) \\
        & = & \Phi_{(\gamma,\chi)}(y) \circ (\gamma \act f )\\
        & = & (\partial (\chi)\gamma \act f) \circ \Phi_{(\gamma,\chi)}(x)
      \end{array}
    \end{equation}
  \end{itemize}
\end{definition}

The last two expressions are the same as those given in the proof of
Theorem \ref{thm:phifunctorial}, written in our new notation. It will
be revisited in (\ref{eq:square-transcomp-target2}).

\subsection{Example: Adjoint Action of 2-Groups} \label{adjoint}

Here we want to describe the adjoint action of a 2-group $\G$ on
itself.  This is the analog of the usual adjoint action of a group $G$
on itself by conjugation. The adjoint action is a functor
\begin{equation}
  \Phi : G \ra End(G)
\end{equation}
given by the property that
\begin{equation}
  \Phi_{\gamma}(g) = \gamma g \gamma^{-1}
\end{equation}

We want a 2-functor given ``by conjugation'', insofar as this makes
sense.  In fact, as we shall see, this is easy to do in the language
of crossed modules, since the axioms (\ref{cm1}) and (\ref{cm2}) for a
crossed module $(G,H,\rhd,\partial)$ imply that the action $\rhd$ of $G$ on 
$H$ resembles conjugation as much as possible. This will be made even clearer by using 
the square calculus in the double groupoid of $\G$, ${\cal D}(\G)$.

In accordance with Definition \ref{def:2grp_action}, we take the
action of $\G$ on itself to be given by a 2-functor from $\G$ to
$\cat{Cat}$ with image in $End(\G)$, i.e. the ``acting'' $\G$ is
regarded as a 2-category, whilst the ``acted on'' $\G$ is regarded as
a (monoidal) category. (This is analogous to the situation for the
adjoint action of a group $G$, regarded as a category, acting on $G$,
regarded as a set).

\begin{definition}
  Suppose $\G$ is the 2-group given by a crossed module
  $(G,H,\rhd,\partial)$.  Then we define a strict 2-functor:
  \begin{equation}
    \Phi : \G \ra \cat{Cat} 
  \end{equation}
  with image in $End(\G)$
  in the following way.  At the object level, $\Phi(\ast)=\G$, where $\G$ on the right is taken to be 
a category. For each morphism $\gamma \in \G^{(1)}$
  \begin{equation}
    \Phi_{\gamma} : \G \rightarrow \G
  \end{equation}
  is a morphism of $End(\G)$, namely an endofunctor of $\G$.  Its object
  map is given by:
  \begin{equation}\label{eq:adjoint-functor-ob}
    \Phi_{\gamma}(g) = \gamma g \gamma^{-1}
  \end{equation}
  and its morphism map is given by
  \begin{equation}\label{eq:adjoint-functor-mor}
    \Phi_{\gamma}(g,\eta) = (\gamma g \gamma^{-1}, \gamma \rhd \eta)
  \end{equation}
  For each 2-morphism $(\gamma, \chi) \in \G^{(2)}$ there is a natural
  transformation from $\Phi_\gamma$ to $\Phi_{\partial(\chi)\gamma}$, given  by:
  \begin{equation}
    \Phi_{(\gamma,\chi)}(g) = (\gamma g \gamma^{-1} , \chi ( \gamma g \gamma^{-1} )\rhd \chi^{-1} ) )
\label{eq:adjoint-nat-trans}
  \end{equation}
  \label{def:adjoint}
\end{definition}

In the action notation of Definition \ref{def:act-notation}, the first
two simply say that $\gamma \act g = \gamma g \gamma^{-1}$ and $\gamma
\act (g,\eta) = (\gamma \act g , \gamma \rhd \eta)$. The third
part of that definition will define $(\gamma,\chi) \act (g,\eta)$,
which is not directly given by $\Phi$. However, we can understand it
using the calculus of squares introduced in Section
\ref{sec:doublegroupoid} for the double groupoid of quintets ${\cal
  D}(\G)$ associated to $\G$.

Represent morphisms $(g_1,\eta)$ in the category $\G$ as squares, so
that the action of $\Phi_\gamma$ on such a square is given by:
\begin{equation}
\xybiglabels \vcenter{\xymatrix@M=0pt@=3pc{\ar@{-} [d] \ar@{-} [r]^{g_1} \ar@{} [dr]|\eta & \ar@{-} [d] \\
\ar@{-} [r]_{g_2}  & }}
\quad 
\stackrel{\Phi_\gamma}{\mapsto}
\quad
\xybiglabels \vcenter{\xymatrix @=3pc @W=0pc @M=0pc { \ar@{-}[r] ^{\gamma} \ar@{-}[d]_{} \ar@{}[dr]|{} & \ar@{-}[r] ^{g_1} \ar@{-}[d]|{}
\ar@{}[dr]|{\eta} & \ar@{-}[r] ^{\gamma^{-1}} \ar@{-}[d]^{}  \ar@{}[dr]|{} & \ar@{-}[d]^{}
\\ \ar@{-}[r] _{\gamma} & \ar@{-}[r] _{g_2} & \ar@{-}[r] _{\gamma^{-1}} &
}} \, = \, 
\xybiglabels \vcenter{\xymatrix@M=0pt@=3pc@C=4pc{\ar@{-} [d] \ar@{-} [r]^{\gamma g_1\gamma^{-1}} \ar@{} [dr]|{\gamma\rhd\eta} & \ar@{-} [d] \\
\ar@{-} [r]_{\gamma g_2\gamma^{-1}}  & }} \, = \, 
\xybiglabels \vcenter{\xymatrix@M=0pt@=3pc@C=4pc{\ar@{-} [d] \ar@{-} [r]^{\gamma \act g_1} \ar@{} [dr]|{\gamma\rhd\eta} & \ar@{-} [d] \\
\ar@{-} [r]_{\gamma \act g_2}  & }}
\label{eq:Phi-1-square}
\end{equation}
where $g_2=\partial(\eta)g_1$. Likewise we can represent the morphism $\Phi_{(\gamma_1,\chi)}(g)$ (which has no direct analog in the action notation) as
\begin{equation}
\Phi_{(\gamma_1,\chi)}(g)
\, = \,
\xybiglabels \vcenter{\xymatrix @=3pc @W=0pc @M=0pc { \ar@{-}[r] ^{\gamma_1} \ar@{-}[d]_{} \ar@{}[dr]|{\chi} & \ar@{-}[r] ^{g} \ar@{-}[d]|{}
\ar@{}[dr]|{} & \ar@{-}[r] ^{\gamma_1^{-1}} \ar@{-}[d]^{}  \ar@{}[dr]|{\chi^{-h}} & \ar@{-}[d]^{}
\\ \ar@{-}[r] _{\gamma_2} & \ar@{-}[r] _{g} & \ar@{-}[r] _{\gamma_2^{-1}} &
}}
\, = \,
\xybiglabels \vcenter{\xymatrix@M=0pt@=3pc@C=8pc{\ar@{-} [d] \ar@{-} [r]^{\gamma_1 \act g} \ar@{} [dr]|{\chi(\gamma_1 \act g )\rhd\chi^{-1}} & \ar@{-} [d] \\
\ar@{-} [r]_{\gamma_2 \act g}  & }}
\label{eq:Phi-2-square}
\end{equation}
where $\gamma_2=\partial(\chi)\gamma_1$. 

It follows from (\ref{eq:Phi-1-square}) that $\Phi_\gamma(g_1,\eta)$ has the correct source and target. $\Phi_\gamma$ is a functor, since it preserves identities, $\Phi_\gamma(g,1_H)=(\gamma g\gamma^{-1},1_H)$ (immediate), and composition, i.e. 
$\Phi_\gamma(g_2,\eta_2) \circ \Phi_\gamma(g_1,\eta_1)= \Phi_\gamma(g_1,\eta_2\eta_1)$, where $g_2=\partial(\eta_1)g_1$,  because of:
$$
\xybiglabels \vcenter{\xymatrix @=3pc @W=0pc @M=0pc { \ar@{-}[r] ^{\gamma} \ar@{-}[d]_{} \ar@{}[dr]|{} & \ar@{-}[r] ^{g_1} \ar@{-}[d]|{}
\ar@{}[dr]|{\eta_1} & \ar@{-}[r] ^{\gamma^{-1}} \ar@{-}[d]^{}  \ar@{}[dr]|{} & \ar@{-}[d]^{}
\\ \ar@{-}[r] |{\gamma} \ar@{-}[d]_{} \ar@{}[dr]|{} & \ar@{-}[r] |{g_2} \ar@{-}[d]|{} \ar@{}[dr]|{\eta_2}  & \ar@{-}[r] |{\gamma^{-1}} \ar@{-}[d]^{}  \ar@{}[dr]|{} & \ar@{-}[d]^{}
\\ \ar@{-}[r] _{\gamma} & \ar@{-}[r] _{g_3} & \ar@{-}[r] _{\gamma^{-1}} &
}}
\, = \,
\xybiglabels \vcenter{\xymatrix @=3pc @W=0pc @M=0pc { \ar@{-}[r] ^{\gamma} \ar@{-}[d]_{} \ar@{}[dr]|{} & \ar@{-}[r] ^{g_1} \ar@{-}[d]|{}
\ar@{}[dr]|{\eta_2\eta_1} & \ar@{-}[r] ^{\gamma^{-1}} \ar@{-}[d]^{}  \ar@{}[dr]|{} & \ar@{-}[d]^{}
\\ \ar@{-}[r] _{\gamma} & \ar@{-}[r] _{g_3} & \ar@{-}[r] _{\gamma^{-1}} &
}}
$$

It then follows immediately from (\ref{eq:Phi-2-square}) that $\Phi_{(\gamma_1,\chi)}(g)$ has the correct source and target. The naturality condition (\ref{eq:Phinatural}), i.e. 
$\Phi_{\gamma_2}(g_1,\eta) \circ   \Phi_{(\gamma_1,\chi)}(g_1)  =  \Phi_{(\gamma_1,\chi)}(g_2)   \circ \Phi_{\gamma_1}(g_1,\eta) $, is the equality:
$$
\xybiglabels \vcenter{\xymatrix @=3pc @W=0pc @M=0pc { \ar@{-}[r] ^{\gamma_1} \ar@{-}[d]_{} \ar@{}[dr]|{\chi} & \ar@{-}[r] ^{g_1} \ar@{-}[d]|{} \ar@{}[dr]|{} & 
\ar@{-}[r] ^{\gamma_1^{-1}} \ar@{-}[d]^{}  \ar@{}[dr]|{\chi^{-h}} & \ar@{-}[d]^{}
\\ \ar@{-}[r] |{\gamma_2} \ar@{-}[d]_{} \ar@{}[dr]|{} & \ar@{-}[r] |{g_1} \ar@{-}[d]|{} \ar@{}[dr]|{\eta}  & \ar@{-}[r] |{\gamma_2^{-1}} \ar@{-}[d]^{}  \ar@{}[dr]|{} & \ar@{-}[d]^{}
\\ \ar@{-}[r] _{\gamma_2} & \ar@{-}[r] _{g_2} & \ar@{-}[r] _{\gamma_2^{-1}} &
}}
\, = \,
\xybiglabels \vcenter{\xymatrix @=3pc @W=0pc @M=0pc { \ar@{-}[r] ^{\gamma_1} \ar@{-}[d]_{} \ar@{}[dr]|{} & \ar@{-}[r] ^{g_1} \ar@{-}[d]|{} \ar@{}[dr]|{\eta} & 
\ar@{-}[r] ^{\gamma_1^{-1}} \ar@{-}[d]^{}  \ar@{}[dr]|{} & \ar@{-}[d]^{}
\\ \ar@{-}[r] |{\gamma_1} \ar@{-}[d]_{} \ar@{}[dr]|{\chi} & \ar@{-}[r] |{g_2} \ar@{-}[d]|{} \ar@{}[dr]|{}  & \ar@{-}[r] |{\gamma_1^{-1}} \ar@{-}[d]^{}  \ar@{}[dr]|{\chi^{-h}} & \ar@{-}[d]^{}
\\ \ar@{-}[r] _{\gamma_2} & \ar@{-}[r] _{g_2} & \ar@{-}[r] _{\gamma_2^{-1}} &
}}
$$
which follows from the interchange law by evaluating the vertical compositions on both sides first. 

The first Ver condition, i.e. $\Phi_{(\gamma_2,\chi_2)}(g) \circ \Phi_{(\gamma_1,\chi_1)}(g) = \Phi_{(\gamma_1,\chi_2\chi_1)}(g)$ where $\gamma_2= \partial(\chi_1)\gamma_1$, follows from the equality:
$$
\xybiglabels \vcenter{\xymatrix @=3pc @W=0pc @M=0pc { \ar@{-}[r] ^{\gamma_1} \ar@{-}[d]_{} \ar@{}[dr]|{\chi_1} & \ar@{-}[r] ^{g} \ar@{-}[d]|{}
\ar@{}[dr]|{} & \ar@{-}[r] ^{\gamma_1^{-1}} \ar@{-}[d]^{}  \ar@{}[dr]|{\chi_1^{-h}} & \ar@{-}[d]^{}
\\ \ar@{-}[r] |{\gamma_2} \ar@{-}[d]_{} \ar@{}[dr]|{\chi_2} & \ar@{-}[r] |{g} \ar@{-}[d]|{} \ar@{}[dr]|{}  & \ar@{-}[r] |{\gamma_2^{-1}} \ar@{-}[d]^{}  \ar@{}[dr]|{\chi_2^{-h}} & \ar@{-}[d]^{}
\\ \ar@{-}[r] _{\gamma_3} & \ar@{-}[r] _{g} & \ar@{-}[r] _{\gamma_3^{-1}} &
}}
\, = \,
\xybiglabels \vcenter{\xymatrix @=3pc @W=0pc @M=0pc @C=4pc{ \ar@{-}[r] ^{\gamma_1} \ar@{-}[d]_{} \ar@{}[dr]|{\chi_2\chi_1} & \ar@{-}[r] ^{g} \ar@{-}[d]|{}
\ar@{}[dr]|{} & \ar@{-}[r] ^{\gamma_1^{-1}} \ar@{-}[d]^{}  \ar@{}[dr]|{(\chi_2\chi_1)^{-h}} & \ar@{-}[d]^{}
\\ \ar@{-}[r] _{\gamma_3} & \ar@{-}[r] _{g} & \ar@{-}[r] _{\gamma_3^{-1}} &
}}
$$
The second Ver condition, i.e. $\Phi_{(\gamma,1_H)}(g) = {\rm Id}_{\Phi_\gamma(g)}$, and the first Hor condition, i.e. $\Phi_{\gamma_1}(\Phi_{\gamma_3}(g)) = \Phi_{\gamma_1\gamma_3}(g)$ and $\Phi_1(g)=g$, are both immediate. Finally the second Hor condition, i.e.
$\Phi_{(\gamma_,\chi_1)}( \Phi_{\gamma_4}(g) ) \circ  \Phi_{\gamma_1}(\Phi_{(\gamma_3,\chi_2)}(g)) = \Phi_{(\gamma_1\gamma_3,\chi_1(\gamma_1\rhd\chi_2))}(g)$, corresponds to the equality:
$$
\xybiglabels \vcenter{\xymatrix @=2.8pc @W=0pc @M=0pc  { \ar@{-}[r] ^{\gamma_1} \ar@{-}[d]_{} \ar@{}[dr]|{}  & \ar@{-}[r] ^{\gamma_3} \ar@{-}[d]_{} \ar@{}[dr]|{\chi_2}  & 
\ar@{-}[r] ^{g} \ar@{-}[d]|{} \ar@{}[dr]|{} & 
\ar@{-}[r] ^{\gamma_3^{-1}} \ar@{-}[d]^{}  \ar@{}[dr]|{\chi_2^{-h}}    & \ar@{-}[r] ^{\gamma_1^{-1}} \ar@{-}[d]^{}  \ar@{}[dr]|{} & \ar@{-}[d]^{}
\\ \ar@{-}[r] |{\gamma_1} \ar@{-}[d]_{} \ar@{}[dr]|{\chi_1} & \ar@{-}[r] |{\gamma_4} \ar@{-}[d]_{} \ar@{}[dr]|{} & 
\ar@{-}[r] |{g} \ar@{-}[d]|{} \ar@{}[dr]|{}  & 
\ar@{-}[r] |{\gamma_4^{-1}} \ar@{-}[d]^{}  \ar@{}[dr]|{} & \ar@{-}[r] |{\gamma_1^{-1}} \ar@{-}[d]^{}  \ar@{}[dr]|{\chi_1^{-h}} & \ar@{-}[d]^{}
\\ \ar@{-}[r] _{\gamma_2} & \ar@{-}[r] _{\gamma_4} & \ar@{-}[r] _{g} & \ar@{-}[r] _{\gamma_4^{-1}} & \ar@{-}[r] _{\gamma_2^{-1}} &
}}
\, = \,
\xybiglabels \vcenter{\xymatrix @=3pc @W=0pc @M=0pc @C=7pc{ \ar@{-}[r] ^{\gamma_1\gamma_3} \ar@{-}[d]_{} \ar@{}[dr]|{\chi_1(\gamma_1\rhd\chi_2)} & \ar@{-}[r] ^{g} \ar@{-}[d]|{}
\ar@{}[dr]|{} & \ar@{-}[r] ^{(\gamma_1\gamma_3)^{-1}} \ar@{-}[d]^{}  \ar@{}[dr]|{(\chi_1(\gamma_1\rhd\chi_2))^{-h}} & \ar@{-}[d]^{}
\\ \ar@{-}[r] _{\gamma_2\gamma_4} & \ar@{-}[r] _{g} & \ar@{-}[r] _{(\gamma_2\gamma_4)^{-1}} &
}}
$$
This follows from evaluating the two 2 by 2 arrays of squares without $g$ on the left hand side. Thus we have proved:

\begin{lemma}
The 2-group adjoint action of Definition \ref{def:adjoint} is a well-defined action in the sense of Definition \ref{def:2grp_action}.
\end{lemma}

\begin{remark}
We can now display the functor $\hat{\Phi}$ of Definition \ref{def:2grp_action-Phi-hat}, in terms of squares, and therefore finish describing this action in the notation of Definition \ref{def:act-notation}. Namely $(\gamma_1,\chi) \act (g_1,\eta) = \hat{\Phi}((\gamma_1,\chi), (g_1, \eta))$ is given by:
$$
\xybiglabels \vcenter{\xymatrix @=3pc @W=0pc @M=0pc { \ar@{-}[r] ^{\gamma_1} \ar@{-}[d]_{} \ar@{}[dr]|{\chi} & \ar@{-}[r] ^{g_1} \ar@{-}[d]|{}
\ar@{}[dr]|{\eta} & \ar@{-}[r] ^{\gamma^{-1}} \ar@{-}[d]^{}  \ar@{}[dr]|{\chi^{-h}} & \ar@{-}[d]^{}
\\ \ar@{-}[r] _{\gamma_2} & \ar@{-}[r] _{g_2} & \ar@{-}[r] _{\gamma_2^{-1}} &
}}
\, = \,
\xybiglabels \vcenter{\xymatrix @=3pc @W=0pc @M=0pc { \ar@{-}[r] ^{} \ar@{-}[d]_{} \ar@{}[dr]|{\chi} & \ar@{-}[r] ^{\gamma_1} \ar@{-}[d]_{} \ar@{}[dr]|{} & 
\ar@{-}[r] ^{g_1} \ar@{-}[d]|{} \ar@{}[dr]|{\eta} & 
\ar@{-}[r] ^{\gamma^{-1}} \ar@{-}[d]^{}  \ar@{}[dr]|{} & \ar@{-}[r] ^{} \ar@{-}[d]^{}  \ar@{}[dr]|{\chi^{-1}} & \ar@{-}[d]^{}
\\ \ar@{-}[r] _{\partial(\chi)} & \ar@{-}[r] _{\gamma_1} & \ar@{-}[r] _{g_2} & \ar@{-}[r] _{\gamma_1^{-1}} & \ar@{-}[r] _{\partial(\chi)^{-1}} &
}}
$$
The 5-square array on the right makes it clear that the 2-group adjoint action of $(\gamma_1, \chi)$ can be regarded as the ordinary adjoint action of $\gamma_1$ on edges, extended to squares labelled by $\gamma_1$ acting on squares, followed by the adjoint action of $\chi$, in the sense of conjugation with the square labelled $\chi$.
Evaluating the array and dropping the indices gives an algebraic formula for this action:
\begin{equation}
(\gamma,\chi) \act (g,\eta) = (\gamma g \gamma^{-1}, \chi (\gamma\rhd \eta) (\gamma g \gamma^{-1})\rhd \chi^{-1})
\label{eq:Phihat-adjointformula}
\end{equation}
\end{remark}

\section{Transformation Double Category for a 2-Group Action}\label{sec:transdoublecat}

In this section, we recall the construction of the transformation
groupoid for a group action on a set, and consider the analogous
construction for a 2-group action on a category.

If we are given an action of a group on a set, there is a groupoid
which corresponds to it.

\begin{definition}Given a group action $\phi : G \ra End(X)$, the
  \defn{transformation} groupoid $X \wquot G$ is the groupoid with:
  \begin{itemize}
  \item \defn{Objects}: $x \in (X \wquot G)^{(0)} = X$
  \item \defn{Morphisms}: $(\gamma,x) \in (X \wquot G)^{(1)} = G
    \times X$, with source and target maps $s(\gamma,x) = x$, and
    $t(\gamma,x) = \phi_\gamma(x)$
  \item \defn{Composition}: $(\gamma', \phi_\gamma(x))\circ (\gamma, x)= (\gamma'\gamma,x)$
  \end{itemize}
\label{def:transfngroupoid}
\end{definition}
It is clear that this is a groupoid, whose morphisms are invertible
since $G$ is a group.  The composition then encodes both the group
multiplication (in the first component) and the action (in the second
component). Formally, the set $P = (X \wquot G)^{(1)} \times_{(X
  \wquot G)^{(0)}} (X \wquot G)^{(1)}$ of composable pairs of
morphisms in $X \wquot G $ is then given by the pullback square:
\begin{equation}
  \xymatrix{
    P \ar[r]^{} \ar[d]_{} & G \times X \ar[d]^{\hat{\phi}} \\
    G \times X \ar[r]_{\pi_X} & X
  }
\end{equation}
There is an obvious commuting diagram with $G\times G\times X$
replacing $P$ in the above diagram, and hence, by the universal
property of this pullback, there is a unique map from $G \times G
\times X$ to $P$, given by:
\begin{equation}
(\gamma',\gamma,x)\mapsto ((\gamma', \phi_\gamma(x)),(\gamma,x))
\end{equation}
which is clearly an isomorphism. In this way, the composition map from
$P$ to $G\times X$ agrees with the map $ m \times Id_X$ in
(\ref{eq:actioncondition-diag}).

In the next subsections, we will develop a similar construction in
$\cat{Cat}$, which will give a transformation \textit{double category}
for a categorical group action. As discussed in Section
\ref{sec:doublegroupoid}, we reserve the term ``double category'' for
a certain symmetric point of view of this structure. To construct it,
however, we use an equivalent definition as a $\cat{Cat}$-category.

\subsection{Construction of the Transformation Cat-Category
  for a 2-Group Action}
\label{sec:transcat-cat}

The most obvious way to define an analog of the transformation
groupoid in the situation of a 2-group action comes by simply
following the same constructions from an ordinary group action,
replacing $G$ with $\G$ and $X$ with $\C$. Thus, one has action
diagrams in $\cat{Cat}$.

Next we will construct a transformation groupoid as we did for group
actions on sets, but it will be internal to $\cat{Cat}$, so we call it a
\textit{transformation $\cat{Cat}$-groupoid}.

Thus, now one has a category
$(\C \wquot \G)^{(0)}$ of objects, and a category $(\C \wquot
\G)^{(1)}$ of morphisms. These, of course, have objects and morphisms
of their own, but this fact is invisible to the construction and
becomes important only when we want to describe the resulting structure
concretely.

\begin{definition} 
  Given a 2-group $\G$, a category $\C$, and an action of $\G$ on $\C$
  as in Definitions \ref{def:2grp_action} and
  \ref{def:2grp_action-Phi-hat}, the transformation $\cat{Cat}$-groupoid $\C
  \wquot \G$ is the groupoid internal to $\cat{Cat}$ with:
  \begin{itemize}
  \item \defn{Category of objects}:  $(\C \wquot \G)^{(0)} = \C$.
  \item \defn{Category of morphisms}: $(\C \wquot \G)^{(1)} = \G
    \times \C$, with 
    \begin{itemize}
    \item \defn{Source} functor $s=\pi_{\C} : \G \times \C \ra \C$
    \item \defn{Target} functor $t = \hat{\Phi} : \G \times \C \ra \C$
    \item \defn{Identity inclusion} functor $e = 1_{\G} \times Id_{\C}
      : \C \ra \G \times \C$
    \item \defn{Inverse} functor $inv = (inv_\G, t) : \G \times \C \ra
      \G \times \C$
    \end{itemize}
  \item \defn{Category of composable pairs}: $\cat{P}$, given by the pullback diagram
    \begin{equation}
      \xymatrix{
        \cat{P} \ar[r]^{} \ar[d]_{} & \G \times \C \ar[d]^{\hat{\Phi}} \\
        \G \times \C \ar[r]_{\pi_{\C}} & \C
      }
    \end{equation}
  \item \defn{Composition functor}: given by the action square
    (\ref{eq:catactioncondition-diag})
  \end{itemize}
\label{def:internalC//G}
\end{definition}

The situation is entirely analogous to the transformation groupoid
for an ordinary group $G$, as is seen easily by considering the effect
on objects. (Thus the inverse functor, on objects, takes $(\gamma,x)
\mapsto (\gamma^{-1}, \gamma \act x)$, just as with an ordinary
transformation groupoid.) Analogously with the situation for groups,
we can think of $(\C \wquot \G)^{(1)}$ as a semidirect product
2-group; and there is a canonical isomorphism between $\G\times \G
\times \C$ and $\cat{P}$. We should check that this definition makes
sense.

\begin{theorem}
  The transformation $\cat{Cat}$-groupoid $\C \wquot \G$ is a well-defined
  groupoid internal to $\cat{Cat}$.
\label{thm:transf-gpd}
\end{theorem}
\begin{proof}
  To confirm this, we must check that the source, target, composition, identity inclusion and inverse
  functors are well defined, and satisfy the usual
  properties for a groupoid, primarily associativity and the left and
  right unit laws.

  The fact that the source and target maps are functors is obvious,
  since they are just the projection from a product, and $\hat{\Phi}$
  respectively. The first is necessarily a functor, while the second
  is a functor by definition. 

Recall our notation:
\begin{equation} 
\gamma \act x= \Phi_\gamma(x) = \hat{\Phi}(\gamma, x) \qquad (\gamma, \chi)\act f = \hat{\Phi} ((\gamma, \chi), f)
\end{equation}
and make explicit the source, target and composition functors at the
morphism level as follows:
\begin{equation}
s((\gamma, \chi), f) =f \qquad t((\gamma, \chi), f)=(\gamma, \chi)\act f 
\end{equation}
\begin{equation}
((\gamma_1, \chi_1), (\gamma_3, \chi_2) \act f) \circ ((\gamma_3, \chi_2),f)=( (\gamma_1\gamma_3, \chi_1(\gamma_1\rhd \chi_2)), f)
\label{eq:explicitcomp}
\end{equation}
(see (\ref{eq:Phihat-ASmorphisms})). The target of the composition is well-defined due to the commutativity
of (\ref{eq:catactioncondition-diag}):
\begin{equation}
(\gamma_1, \chi_1) \act ((\gamma_3, \chi_2) \act f) = (\gamma_1\gamma_3, \chi_1(\gamma_1\rhd \chi_2))\act f
\label{eq:targetcomposition}
\end{equation} 

To show the functoriality of composition, we refer to the following array of squares in ${\cal D}(\G)$, underlying the calculation:
\begin{equation}
\xybiglabels \vcenter{\xymatrix @=3pc @W=0pc @M=0pc { \ar@{-}[r] ^{\gamma_1} \ar@{-}[d]_{} \ar@{}[dr]|{\chi_1} & \ar@{-}[r] ^{\gamma_3} \ar@{-}[d]|{}
\ar@{}[dr]|{\chi_2} &  \ar@{-}[d]^{}
\\ \ar@{-}[r] |{\gamma_2} \ar@{-}[d]_{} \ar@{}[dr]|{\chi_3} & \ar@{-}[r] |{\gamma_4} \ar@{-}[d]|{} \ar@{}[dr]|{\chi_4}  & \ar@{-}[d]^{}
\\ \ar@{-}[r] _{\gamma_5} & \ar@{-}[r] _{\gamma_6}  &
}}
\label{diag:intlaw}
\end{equation}
We denote composition in $\G\times \C$ by $\bar{\circ}$ and composition in $\C$ by $\circ_{\C}$, to distinguish them from the composition functor $\circ$.
Then we have:
\begin{eqnarray*}
((\gamma_4,\chi_4),g)\,\bar{\circ}\, ((\gamma_3,\chi_2),f)&  = & ((\gamma_3, \chi_4\chi_2),g\circ_{\C}f) \\
((\gamma_2,\chi_3),(\gamma_4,\chi_4)\act g)\,\bar{\circ}\,((\gamma_1,\chi_1),(\gamma_3,\chi_2)\act f)&  = &
((\gamma_1, \chi_3\chi_1),((\gamma_4,\chi_4)\act g) \circ_{\C} \\
& & ((\gamma_3,\chi_2)\act f)) \\
& = & ((\gamma_1, \chi_3\chi_1),(\gamma_3, \chi_4\chi_2)\act (g\circ_{\C}f))
\end{eqnarray*}
where we use (\ref{eq:Phihat-func1}) in the final equation. We also have:
\begin{eqnarray*}
((\gamma_2,\chi_3),(\gamma_4,\chi_4)\act g)\,{\circ}\,((\gamma_4,\chi_4), g)&  = & ((\gamma_2\gamma_4, \chi_3(\gamma_2\rhd \chi_4)),g) \\
((\gamma_1,\chi_1),(\gamma_3,\chi_2)\act f)\,{\circ}\,((\gamma_3,\chi_2), f)&  = & ((\gamma_1\gamma_3, \chi_1(\gamma_1\rhd \chi_2)),f) \\
\end{eqnarray*}
Thus it remains to show the equality of two expressions:
$$
((\gamma_2\gamma_4, \chi_3(\gamma_2\rhd \chi_4)),g)\, \bar{\circ} \, ((\gamma_1\gamma_3, \chi_1(\gamma_1\rhd \chi_2)),f)
= ((\gamma_1\gamma_3, \chi_3(\gamma_2\rhd \chi_4)  \chi_1(\gamma_1\rhd \chi_2)),g\circ_{\C}f) 
$$
and
$$
((\gamma_1, \chi_3\chi_1),(\gamma_3, \chi_4\chi_2)\act (g\circ_{\C}f)) \circ ((\gamma_3, \chi_4\chi_2),g\circ_{\C}f) = 
((\gamma_1\gamma_3, \chi_3\chi_1 \gamma_1\rhd (\chi_4\chi_2) ),g\circ_{\C}f) .
$$
This equality is immediate from the interchange law in ${\cal D}(\G)$ applied to the array (\ref{diag:intlaw}).  

Associativity of the composition functor follows from the fact that the horizontal composition
of three squares in $\cal{D}(\G)$ has a unique evaluation:
\begin{equation}
\xybiglabels \vcenter{\xymatrix @=3pc @W=0pc @M=0pc { \ar@{-}[r] ^{\gamma_1} \ar@{-}[d]_{} \ar@{}[dr]|{\chi_1} & \ar@{-}[r] ^{\gamma_3} \ar@{-}[d]|{}
\ar@{}[dr]|{\chi_2} & \ar@{-}[r] ^{\gamma_5} \ar@{-}[d]^{}  \ar@{}[dr]|{\chi_3} & \ar@{-}[d]^{}
\\ \ar@{-}[r] _{\gamma_2} & \ar@{-}[r] _{\gamma_4} & \ar@{-}[r] _{\gamma_6} &
}}
\label{eq:assoc-3squares}
\end{equation}
The identity inclusion $e$ is clearly a functor, and satisfies:
\begin{eqnarray*}
((\gamma, \chi),  f) \circ ((1_G,1_H),f) & = & ((\gamma, \chi),  f) \\
((1_G,1_H), (\gamma, \chi) \act f) \circ ((\gamma, \chi),f) & = & ((\gamma, \chi),f)
\end{eqnarray*}
where in the first equation we have a composable pair, since (if the
target of $f$ is $y$):
$$
t((1_G,1_H),f) = (1_G,1_H)\act f = \Phi_{(1_G,1_H)}(y) \circ \Phi_{1_G}(f)=\Phi_{(1_G,1_H)}(y) \circ f = f
$$ using (\ref{eq:PhihatfromPhi}) in the second equality,
functoriality of $\Phi$ in the third equality and property F2 (see
Def. \ref{def:2grp_action} ) in the final equality. 

Finally, the inverse functor assigns, to every $((\gamma, \chi), f)$
in the category of morphisms, its inverse $((\gamma^{-1}, \chi^{-h}),
(\gamma, \chi)\act f)$, where we note that $(\gamma^{-1},
\chi^{-h})= 
(\gamma^{-1}, (\partial(\chi)\gamma)^{-1}\rhd \chi^{-1}) =
(\gamma^{-1}, \gamma^{-1}\rhd \chi^{-1})$. This follows
immediately from (\ref{eq:explicitcomp}). Functoriality of the inverse functor follows 
easily from (\ref{eq:Phihat-func1}) and functoriality of the horizontal inverse in 
${\cal D}(\G)$.
\end{proof}

This construction
of a $\cat{Cat}$-category has a slight drawback, which is the apparent asymmetry of its definition. In the next section we will address this issue.

\subsection{The Transformation Double Category}
\label{sec:transdouble}

The construction
we have given in the previous section naturally produces a groupoid
internal to $\cat{Cat}$. However, the view of a double category as an
internal category in $\cat{Cat}$ obscures the underlying symmetry of
this structure, and the structure in the ``transverse'' direction to
the categories of objects and morphisms.

As we remarked in Section \ref{sec:doublegroupoid},
$\cat{Cat}$-categories are equivalent to a structure defined in a more
symmetric way, having two types of morphisms (horizontal and
vertical), and squares. The properties can be deduced from the
equivalence: for example, the ``interchange law'', which says that mixed
horizontal and vertical composites can be taken in any order, amounts
to the functoriality of the composition $\circ$.

Now, the definition of double categories is symmetric under exchanging the
roles of ``horizontal'' and ``vertical'' morphisms in our diagrams of
squares. One can reflect all squares in a diagonal, and obtain another
double category, which we call the \textit{transpose} of the first
double category. The resulting structure, naturally, still has an interpretation
as a category internal to $\cat{Cat}$.

More technically, we have the following:

\begin{definition}
  If ${\cal D}$ is an internal category in $\cat{Cat}$, then let the
  \defn{transpose} of ${\cal D}$, which we denote $\widetilde{\cal
    D}$, be the internal category in $\cat{Cat}$ defined by the
  following:
  \begin{itemize}
  \item The category of objects has
    \begin{itemize}
    \item Objects: the objects of the category  ${\cal D}^{(0)}$
    \item Morphisms: the objects of the category ${\cal D}^{(1)}$
    \end{itemize}
  \item The category of morphisms has
    \begin{itemize}
    \item Objects: the morphisms of the category ${\cal D}^{(0)}$
    \item Morphisms: the morphisms of the category ${\cal D}^{(1)}$
    \end{itemize}
  \item The identity inclusion, source and target, and composition
    functors $(\widetilde{e}, \widetilde{s}, \widetilde{t},
    \widetilde{\circ})$ have as object maps the corresponding
    structure maps from the category ${\cal D}^{(0)}$, and as morphism
    maps the corresponding structure maps from the category ${\cal
      D}^{(1)}$
  \end{itemize}
\end{definition}

This is a double-category analog of the operation of taking the
``opposite'' of a category: another category in which morphisms are
taken to be oriented in the opposite direction. Indeed, together with
such opposite operations in both horizontal and vertical directions,
the transpose generates a whole group of operations which take one
double category to another: it is plainly isomorphic to the dihedral
group $D_4$, the symmetries of a generic square, since each is
determined by the source vertex of a generic square, together with the
sense of horizontal and vertical. The opposites do not illustrate much
new structure, however, so we will restrict our attention to the
transpose.

The fact that the transpose is again an internal category in
$\cat{Cat}$ is a standard consequence of basic facts about
internalization. We do not want to assume all readers are accustomed
to such internal constructions, so will sketch a proof 
to convey the essential idea.

\begin{lemma}\label{lemma:transposedoublecat}
  If ${\cal D}$ is a category internal in $\cat{Cat}$, so is
  $\widetilde{\cal D}$.
\end{lemma}
\begin{proof}
  First, one must check that the structure maps $(\widetilde{e}, \widetilde{s}, \widetilde{t},
    \widetilde{\circ})$ are
  functors. This follows from the compatibility conditions for ${\cal
    D}$. We give the example of the source functor $\widetilde{s}$ here: the
  others follow similar lines.
  
  We claim that the source functor in $\widetilde{\cal D}$
  \begin{equation}
    \widetilde{s} : \widetilde{\cal D}^{(1)} \ra \widetilde{\cal D}^{(0)}
  \end{equation} preserves composition, that is, for two composable morphisms $\chi_1$ and $\chi_2$ in $\widetilde{\cal D}^{(1)}$
  \begin{equation}
    \widetilde{s}(\chi_1 \circ \chi_2) = \widetilde{s}(\chi_1) \circ
    \widetilde{s}(\chi_2)
  \end{equation}
  where the compositions are in $\widetilde{\cal D}^{(1)}$ and
  $\widetilde{\cal D}^{(0)}$ respectively. Interpreted in ${\cal D}$,
  this is the statement that the composition functor $\circ : {\cal
    D}^{(1)} \times_{{\cal D}^{(0)}} {\cal D}^{(1)} \ra {\cal
    D}^{(1)}$ is compatible with the source maps in ${\cal D}^{(1)}$
  and ${\cal D}^{(0)}$, which is just part of the fact that $\circ$ is
  a functor. Likewise $\widetilde{s}$ is compatible with the source,
  target, and identity inclusion maps in ${\cal D}^{(1)}$ and ${\cal
    D}^{(0)}$. Similar arguments prove the functoriality and
  compatibility of $\widetilde{t}$, $\widetilde{e}$, and
  $\widetilde{\circ}$. So all the structure maps
  $(\widetilde{e},\widetilde{s},\widetilde{t},\widetilde{\circ})$ are
  well-defined morphisms in $\cat{Cat}$.

  Furthermore, since the object and morphism maps of
  $(\widetilde{e},\widetilde{s},\widetilde{t},\widetilde{\circ})$
  satisfy the axioms for a category separately by definition, the
  functors satisfy them as well. Since its structure maps are well
  defined and satisfy all the axioms for a category, $\widetilde{\cal
    D}$ is indeed an internal category in $\cat{Cat}$.
\end{proof}

In the above sketch, we deliberately chose to compare the interaction
of different maps in the two directions, namely source and
composition, to emphasize the differences. It is worth remarking that
some conditions are symmetrical: for example, the fact that $\circ$ is
functorial, and therefore preserves composition, corresponds in the
transpose to exactly the same fact about $\widetilde{\circ}$. This is
a form of the \textit{interchange law}.

The above is a special case of a more general duality which can occur
with internalization. Suppose we have any two types of objects, $X$
and $Y$, which can be defined as collections of objects and morphisms
satisfying certain diagrammatic axioms. Then it is equivalent to
define $X$-objects internal to the category of $Y$-objects, and
$Y$-objects internal to the category of $X$-objects.

Next, we want use the transpose to better understand the structure of
the transformation $\cat{Cat}$-groupoid $\C \wquot \G$ associated to a
2-group action. We state the structure of the transpose of
$\C \wquot \G$ as a theorem. The essential new observation is that the transpose
$\widetilde{\C \wquot \G}$ is built from transformation groupoids
associated to group actions in the usual sense.

\begin{theorem}\label{thm:transposetransform}
  Given a 2-group $\G$, a category $\C$, and an action of $\G$ on $\C$
  given by $\hat{\Phi}: \G \times \C \ra \C$ in $\cat{Cat}$, the transpose of $\C \wquot \G$
  is the category $\widetilde{\C \wquot \G}$ internal to $\cat{Gpd}$, with:
  \begin{itemize}
  \item Groupoid of Objects: $(\widetilde{\C \wquot \G})^{(0)}$,
    equal to the transformation groupoid $\C^{(0)} \wquot
    \G^{(0)}$ associated to the action given by $\hat{\Phi}^{(0)}$.
  \item Groupoid of Morphisms: $(\widetilde{\C \wquot \G})^{(1)}$, equal to the transformation
    groupoid $\C^{(1)} \wquot \G^{(1)}$ associated to the action given
    by $\hat{\Phi}^{(1)}$.
  \item Identity inclusion  
    functor 
    \begin{equation}
      \widetilde{e} : \C^{(0)} \wquot \G^{(0)} \ra \C^{(1)} \wquot \G^{(1)}
    \end{equation}
    source and target functors
    \begin{equation}
      \widetilde{s},\widetilde{t} : \C^{(1)} \wquot \G^{(1)} \ra \C^{(0)} \wquot \G^{(0)}
    \end{equation}
    and composition functor
    \begin{equation}
      \widetilde{\circ} : (\C^{(1)} \wquot \G^{(1)}) \times_{\C^{(0)} \wquot \G^{(0)}} (\C^{(1)} \wquot \G^{(1)}) \ra \C^{(1)} \wquot \G^{(1)}
    \end{equation}
    are the functors whose object maps are the corresponding
    $(e,s,t,\circ)$ for $\C$ and whose morphism maps are those for $\G
    \times \C$.
  \end{itemize}
\end{theorem}
\begin{proof}  Since this is a special case of the transpose construction of Lemma \ref{lemma:transposedoublecat}, we just need to
show the equalities for $(\widetilde{\C \wquot \G})^{(0)}$ and $(\widetilde{\C \wquot \G})^{(1)}$.
  First, consider the action as internal to $\cat{Cat}$, by the
  inclusion of $\cat{Gpd}$ in $\cat{Cat}$. Then the case of the groupoid of objects is
  clear: its objects are the objects of $x \in \C$, and its morphisms
  are objects of $(\C \wquot \G)^{(1)} = \G \times \C$, i.e.
  they are labeled by pairs $(\gamma,x) \in \G^{(0)} \times
  \C^{(0)}$, and $(\gamma,x)$, as a morphism in $(\widetilde{\C \wquot \G})^{(0)}$, has source $x$ and target  $\gamma \act x$.
  Composition is determined by the object map of the composition functor 
  in $\C \wquot \G$ (\ref{eq:explicitcomp}) , i.e. by the monoidal product
  $\otimes$ of $\G^{(0)}$. But this is exactly the transformation groupoid
  $\C^{(0)} \wquot \G^{(0)}$, where $\G^{(0)}$ is a group with product
  $\otimes$, acting on $\C^{(0)}$ by $\hat{\Phi}^{(0)}$.

  The argument for the groupoid of morphisms $(\widetilde{\C \wquot \G})^{(1)}$ is substantially the
  same. Its objects are the morphisms of $\C$, i.e. labelled by $f\in \C^{(1)}$, and its morphisms 
  are pairs $((\gamma,\chi),f)$ with source $f$ and target  $(\gamma,\chi)\act f$. Composition in 
  the groupoid of morphisms is determined by the morphism part of the composition functor in  ${\C \wquot \G}$
  (\ref{eq:explicitcomp}).

Since $(\widetilde{\C \wquot \G})^{(0)}$ and $(\widetilde{\C \wquot \G})^{(1)}$ are in fact groupoids, and all structure maps, being functors,
are also morphisms of $\cat{Gpd}$, we have that $\widetilde{\C \wquot \G}$ is a category internal to  $\cat{Gpd}$.
\end{proof}

We would like to emphasize again, that in the transposed perspective of $\widetilde{\C \wquot \G}$, both the category of objects and the category of morphisms are
themselves transformation groupoids in the usual sense, associated to group actions. 

We are now ready to combine the perspectives given by $\C \wquot \G$ and
its transpose to express the whole local symmetry structure implied in the
action of $\G$ on $\C$ in a balanced way.  This
balanced viewpoint is expressed in our definition below of the transformation
double category $\C \wquot \G$ (abusing notation somewhat by keeping the same name
as for the internal category). We recall from subsection \ref{sec:doublegroupoid} 
the notation $O$ for the set of objects, $\cal H$ and $\cal V$ for the sets of
horizontal and vertical morphisms, and $S$ for the set of squares. By the preceding discussion, the
horizontal category is $(\C \wquot \G)^{(0)}=\C$ (Def. \ref{def:internalC//G}), the vertical
category is $(\widetilde{\C \wquot \G})^{(0)}=\C^{(0)} \wquot \G^{(0)}$ (Thm. \ref{thm:transposetransform}), 
and for squares, the
horizontal composition of squares in $S$ is the morphism map of $\circ$, while vertical
composition  of squares in $S$ is the morphism map of $\widetilde{\circ}$.  

\begin{definition}\label{def:action-double-gpd}
  Given an action of a 2-group $\G$ on a category $\C$, in terms of
  $\Phi$ or $\hat{\Phi}$ of definitions \ref{def:2grp_action} and
  \ref{def:2grp_action-Phi-hat}, the transformation double category
  $\C \wquot \G$ is given by:
\begin{itemize}
\item \textbf{Objects}: objects $x \in O = \C^{(0)}$ 
\item \textbf{Horizontal Category}: $(\C \wquot \G)^{(0)} = \C$, with morphisms
  $x\stackrel{f}{\rightarrow}y \in H = \C^{(1)}$ composed as usual
\item \textbf{Vertical Category}: $(\widetilde{\C \wquot \G})^{(0)} = \C^{(0)} \wquot \G^{(0)}$, with
  morphisms displayed as
  $x\stackrel{(\gamma,x)}{\longrightarrow}\gamma\act x \in V =
  (\C^{(0)} \wquot \G^{(0)})^{(1)}$ with source and target as shown and
  composed by $(\gamma', \gamma\act x)\circ (\gamma,x)=
  (\gamma'\gamma,x)$

\item \textbf{Squares}: $S = (\C \wquot \G)^{(1)} = (\widetilde{\C \wquot \G})^{(1)} = \G^{(1)} \times \C^{(1)}$,
    consisting of pairs of morphisms from $\G$ and $\C$, with elements
    of $S$ denoted by $\squaremor{ (\gamma, \chi), f }$, displayed as
  \begin{equation}
    \xymatrix@C=7.5pc@R=4pc {
      x \ar[r]^{f} \ar[d]_{(\gamma, x)} \drtwocell<\omit>{\omit *+[F]{(\gamma,\chi),f}}  &  y \ar[d]^{(\partial(\chi)\gamma,y)}  \\
      {\gamma \act x} \ar[r]_{(\gamma,\chi)\act f}  & {(\partial(\chi)\gamma)\act y}
    }
  \label{eq:squaredef}   
  \end{equation}
  with horizontal and vertical source and target as shown in the
  display. Horizontal and vertical composition (equivalent to
  pasting of diagrams of the form (\ref{eq:squaredef})) is given
  by:
  \begin{equation}\label{eq:squarehorizcomp}
    \squaremor{(\gamma_2, \chi_2),g } 
    \, \circ_h \, \squaremor{ (\gamma_1, \chi_1),f } 
    \, = \, \squaremor{ (\gamma_1, \chi_2\chi_1),g\circ f }
  \end{equation}
  where $\gamma_2 = \partial(\chi_1)\gamma_1$, and:
  \begin{equation}\label{eq:squarevertcomp}
    \squaremor{ (\gamma_1, \chi_1),(\gamma_3,\chi_2)\act f } 
    \, \circ_v \, \squaremor{ (\gamma_3, \chi_2),f } 
    = \squaremor{ (\gamma_1\gamma_3, \chi_1(\gamma_1\rhd\chi_2)),f }
  \end{equation}
  where the underlying squares in ${\cal D}(\G)$ are given in (\ref{eq:F2squares}).

\end{itemize}
\label{def:trans-dbl-cat}
\end{definition}

\begin{theorem}
  ${\C \wquot \G}$  is a well-defined double category.
\label{thm:trans-dbl-cat}
\end{theorem}
\begin{proof}
  This follows from general considerations, but can also be checked by using the properties of the internal categories 
  ${\C \wquot \G}$ and $\widetilde{\C \wquot \G}$, e.g. the vertical target of the
  vertical composition of squares is well-defined due to (\ref{eq:targetcomposition}). We note that
  the vertical target of
  the horizontal composition of squares, equals the composition of the
  vertical targets of the two squares, i.e.
  \begin{equation}
    (\gamma_1, \chi_2\chi_1)\act (g\circ f) = ((\gamma_2, \chi_2)\act g)\circ ((\gamma_1, \chi_1)\act f)
  \end{equation}
  by the functoriality of $\hat{\Phi}$ in (\ref{eq:Phihat-func1}). Likewise the horizontal target of
  the vertical composition of squares equals the composition of the
  horizontal targets of the two squares, i.e.
  \begin{equation}
    (\partial(\chi_1)\gamma_1\partial(\chi_2)\gamma_3,y)= (\partial(\chi_1(\gamma_1\rhd \chi_2))\gamma_1\gamma_3,y).
  \end{equation}
  This follows in a straightforward fashion from crossed module properties.

  Associativity of the horizontal and vertical composition of squares is an immediate consequence of the associativity
  of the composition functors in ${\C \wquot \G}$ and $\widetilde{\C \wquot \G}$.
   Finally the interchange
  law for horizontal and vertical composition of squares corresponds to the proof of functoriality of the composition
  functor in ${\C \wquot \G}$ (Theorem \ref{thm:transf-gpd}) and reduces to stating the interchange law for four squares 
  (\ref{diag:intlaw}) in ${\cal D}(\G)$. 
\end{proof}

\vskip 0.2 cm

\begin{remark}
If we display the squares of the double
  category as follows:
  \begin{equation}   
    \xymatrix@C=7.5pc@R=4pc {
      x \ar@{-->}[r]^{f} \ar@{.>}[d]_{(\gamma, x)} \ar@{} [dr]|{((\gamma,\chi),f)}  &  y \ar@{.>}[d]^{(\partial(\chi)\gamma,y)}  \\
      {\gamma \act x} \ar@{-->}[r]_{(\gamma,\chi)\act f}  & {(\partial(\chi)\gamma)\act y}
    } 
  \end{equation}
with horizontal arrows dashed and vertical arrows dotted, this can be used to illuminate the two internal category perspectives for the double category.
  
From the perspective of the internal category ${\C \wquot \G}$ of
  Definition \ref{def:internalC//G}, these squares of the double
  category 
  display the action of the source and target
  functors $s,\, t$ on objects $(\gamma,x)$,
  $(\partial(\chi)\gamma,y)$, and morphisms $((\gamma,\chi),f)$. The 
  dashed arrows denote the image of the square under source and target
  functors, while the dotted arrows denote its source and target
  internal to the category of morphisms. 

Similarly, from the
  perspective of the internal category $\widetilde{\C \wquot \G}$ of
  Theorem \ref{thm:transposetransform}, these squares 
  display the action of the source and target
  functors $\widetilde{s},\, \widetilde{t}$ on objects $f$,
  $(\gamma,\chi)\act f$, and morphisms
  $((\gamma,\chi),f)$. Now the 
  dotted arrows denote the image of the square under source and target
  functors, while the dashed arrows denote its source and target
  internal to the category of morphisms. 

The squares of the double category superimpose the two
  perspectives giving a balanced view, which makes both internal
  structures transparent at the same time.
\end{remark}

\subsection{Structure of the Transformation Double Category}\label{sec:struc-trans-double-cat}

Next, we will consider some consequences of the construction of $\C
\wquot \G$. As constructed, the main information which can be read
directly from $\C \wquot \G$, viewed as an internal category in $\cat{Cat}$, is precisely the target map, which
is just the action $\hat{\Phi}$ itself. Some less obvious
consequences are more easily read from the transpose $\widetilde{\C
  \wquot \G}$, or from $\C \wquot \G$ seen as a double category.

\subsubsection{Relation to Ordinary Transformation Groupoids}

We have seen, in Theorem \ref{thm:transposetransform}, that the groupoid of objects and the groupoid of
morphisms of $\widetilde{\C \wquot \G}$ are both transformation
groupoids. The structure maps of $\widetilde{\C \wquot \G}$ relate
these groupoids to each other. Being functors, these maps consist of
two parts, which we can consider separately. Considering the
identity-inclusion functor tells us how several transformation
groupoids are nested inside one another, and will justify the notation
$\act$ we previously introduced for our 2-group action. First, note
what these groupoids are in crossed module notation.

\begin{corollary}\label{cor:transgpd-obj}
  If $\G$ is given by a crossed module $(G,H,\partial,\rhd)$, then
  $(\widetilde{\C \wquot \G})^{(0)} = \C^{(0)} \wquot G$, and $\C^{(1)} \wquot G
  \subset (\widetilde{\C \wquot \G})^{(1)} = \C^{(1)} \wquot (G \ltimes H)$. 
\end{corollary}
\begin{proof}
  In the construction of a categorical group from a crossed module,
  the group $G \ltimes H$ is the group of morphisms, and the group $G$
  of objects occurs as the subgroup of elements of the form $(\gamma,1_H)$. The groupoid
  inclusion then follows immediately from the identity-inclusion
  functor in Theorem \ref{thm:transposetransform}.
\end{proof}

The fact that the structure maps are functors means that, from an
``external'' point of view (that is, if we look at underlying sets),
$\C \wquot \G$ combines three closely related group actions in the
standard sense. These three group actions have all been denoted by the
symbol $\act$ introduced in Definition \ref{def:act-notation}. This
notation can now be justified because they are all restrictions of the
action of $\G^{(1)}$ on $\C^{(1)}$, along the identity inclusion maps
for $\G$ and $\C$. In particular, the associated action groupoids are
also related by these restrictions.  More precisely, we have:

\begin{corollary}\label{cor:transgpds-3actions}
  The identity-inclusion functor $\widetilde{e}$ factors into two
  inclusions:
  \begin{equation}
    (\widetilde{\C \wquot \G})^{(0)} \subset \C^{(1)} \wquot \G^{(0)} \subset (\widetilde{\C \wquot \G})^{(1)}
  \end{equation}
  where the first inclusion is given by the identity inclusion of
  $\C$, and the second by the identity inclusion of $\G$.
  \end{corollary}
\begin{proof}
  The fact that objects of $\G$ give endofunctors of $\C$ means that
  the morphism maps $\Phi_\gamma^{(1)}$ determine an action of
  $\G^{(0)}$ on $\C^{(1)}$. Thus, objects of $\G$ act on both
  $\C^{(0)}$ and $\C^{(1)}$. By functoriality, the inclusion of
  objects of $\C$ as identity morphisms in $\C^{(1)}$ is compatible
  with the action. So the subset $\C^{(0)} \subset \C^{(1)}$ is closed
  under the action of G. The associated transformation groupoid has
  as its objects the morphisms $f \in \C$, and as morphisms pairs
  $(\gamma,f)$ which compose in the usual way. These correspond to the
  special squares $\squaremor{(\gamma,1_H),f}$. There is a full
  inclusion of transformation groupoids induced by the inclusion of
  their sets of objects $\C^{(0)} \subset \C^{(1)}$.
  
  The group $\G^{(1)} \cong G \ltimes H$ also acts on $\C^{(1)}$. This
  action is given by $\hat{\Phi}$, and is related to the action of $G$
  on $\C^{(1)}$ by the identity inclusion map for $\G$, namely $G
  \subset G \ltimes H$. In particular, the action of $\G^{(1)}$ on
  $\C^{(1)}$ combines both natural transformations acting on objects,
  and functors acting on morphisms. To see this, note that the squares
  in the double category $\C \wquot \G$ occur by construction within a
  commuting cube which shows the naturality square associated to $f$,
  or equivalently the two ways of expressing $(\gamma,\chi)\act f$ in
  terms of $\Phi$ - see (\ref{eq:PhihatfromPhi}):
  \begin{equation}\label{eq:triplecat-3cell}
    \xymatrix{
      x \ar@{-->}[drrr] \ar[ddd]_{(\gamma,x)}="00" \ar@{=}[dr] \ar[rr]^{f}  & & y \ar[ddd]^{(\gamma,y)}="10" \ar@{=}[dr] & \\
      & x \ar[ddd]_(.4){(\partial (\chi) \gamma,x)}="01" \ar[rr]_(.4){f}  & & y \ar[ddd]^{(\partial (\chi) \gamma,y)}="11"\\
      & & & \\
      \gamma \act x \ar@{-->}[drrr] \ar[dr]_{\Phi_{(\gamma,\chi)}(x)} \ar[rr]^(.6){\gamma \act f} & \uurtwocell<\omit>{\omit *+[F]{((\gamma,\chi),f)}} & \gamma \act y \ar[dr]^{\Phi_{(\gamma,\chi)}(y)} & \\
      & \partial (\chi)\gamma \act x \ar[rr]_{\partial (\chi)\gamma \act f} & & \partial (\chi)\gamma \act y
    }   
  \end{equation}

  So the action of morphisms of $\G$ on morphisms of $\C$ involves
  natural transformations at the source and target objects, as well as
  functors acting on the morphisms themselves. The outside faces of
  the cube (\ref{eq:triplecat-3cell}) are themselves special cases of
  squares in which one of these two parts is an identity. The case
  $\chi = 1_H$, so that $(\gamma,\chi) = Id_{\gamma}$ (for example, the
  front or rear faces of that cube) is precisely the second inclusion
  of the theorem.

  Restricting further to the image of the unit inclusion for $\C$, we
  have the case $f = Id_x$. In particular, $\gamma \act Id_x =
  Id_{\gamma \act x}$ is just a consequence of the fact that the
  action is functorial. (For this reason, it is immediate if we think
  of the action in terms of $\hat{\Phi}$.)

  The action by the subgroup $G$ gives a transformation groupoid with
  fewer morphisms than the action by the big group $G \ltimes H$ - but
  the former is strictly a sub-groupoid (by a non-full inclusion) of
  the latter.
\end{proof}

Notice that this factorization of $\widetilde{e}$ is given
by restricting to the image of the unit inclusions of $\C$ and then
$\G$, and not the other way around. There is no well-defined action of
$\G^{(1)}$ on $\C^{(0)}$. This can easily be seen, again by taking a
special case of a square of the form (\ref{eq:triplecat-3cell}) with
$f=Id_x$. Since $\gamma \act x$ and $\partial (\chi) \gamma \act x$ need
not be the same object, and even in the case they are,
$\Phi_{(\gamma,\chi)}(x)$ need not be an identity, the image of $e :
\C^{(0)} \ra \C^{(1)}$ is not fixed by the action of $\G^{(1)}$ on
$\C^{(1)}$.

\subsubsection{Image of the Composition Functor}

Next, consider the composition $\widetilde{\circ}$. It is somewhat
more complicated than $\circ$, and derives from the 2-group structure
of $\G$, just as composition does in transformation groupoids. A
little example calculation with this composite will show how 2-group
structure and an action must cohere.

Indeed, composition in $(\widetilde{\C \wquot \G})^{(0)}$ is exactly
the composition in the transformation groupoid $\C^{(0)} \wquot G$, by
Corollary \ref{cor:transgpd-obj}. Similarly, the composition in
$(\widetilde{\C \wquot \G})^{(1)}$ is the composition in $\C^{(1)}
\wquot \G^{(1)}$, which is determined by the group multiplication in
$\G^{(1)} \cong G \ltimes H$ (in the crossed-module representation),
together with the action on $\C^{(1)}$. So we can write this as:
\begin{equation}\label{eq:square-transcomp}
  \squaremor{ (\gamma', \chi'),(\gamma,\chi)\act f } 
  \, \widetilde{\circ} \, \squaremor{ (\gamma, \chi),f } 
  = \squaremor{ (\gamma'\gamma, \chi'(\gamma'\rhd\chi)),f }
\end{equation}
We know by construction and the above theorems that this makes a
well-defined double category. It is, however, not given directly by
the action $\Phi$ seen as a 2-functor. In that point of view, the
action of morphisms of $\G$ on morphisms of $\C$ is a consequence of
the naturality of $\Phi_{(\gamma,\chi)}$ for all $(\gamma,\chi) \in
\G^{(1)}$.

Writing the target of this composite concretely, it is:
\begin{equation}\label{eq:square-transcomp-target1}
  (\gamma'\gamma, \chi'(\gamma' \rhd \chi)) \act f
\end{equation}
However, we have concrete expressions for $(\gamma,\chi)\act f$,
shown in the cube (\ref{eq:triplecat-3cell}), namely:
\begin{equation}\label{eq:square-transcomp-target2}
  \begin{array}{ccccc}
    (\gamma,\chi) \act f & = & ( \partial (\chi) \gamma \act f ) & \circ & \Phi_{(\gamma,\chi)}(x) \\
  & = & \Phi_{(\gamma,\chi)}(y) & \circ & (\gamma \act f)
    \end{array}
\end{equation}
where composition is taken in $\C$. Recall that this is a consequence
of the naturality of the bottom square of (\ref{eq:triplecat-3cell})
when we apply $\Phi_{(\gamma,\chi)}$ at the morphism $f$.

Now, if we then take $(\gamma,\chi) \act f$ as the source of the new
square $\squaremor{ (\gamma', \chi'),(\gamma,\chi)\act f }$ to find
its target, we are applying $\Phi_{(\gamma',\chi')}$ at the morphism
$(\gamma,\chi) \act f$. To compute this explicitly, note that each
morphism in the commuting square at the bottom of
(\ref{eq:triplecat-3cell}) gives rise to another commuting square,
and these form four faces of a cube whose other two faces are the image of
the original commuting square under the functors $\Phi_{\gamma'}$ and
$\Phi_{\partial \chi' \gamma'}$ respectively:
\begin{equation}\label{eq:compositesquare-target}
  \xymatrix{
    \partial (\chi') \gamma' \gamma \act x \ar[ddd]_{\partial (\chi') \gamma' \act \Phi_{(\gamma,\chi)}(x)} \ar[rr]^{\partial (\chi')\gamma' \gamma \act f}  & & \partial (\chi' \gamma' \gamma) \act y \ar[ddd]^(.6){\partial (\chi') \gamma' \act \Phi_{(\gamma,\chi)}(y)}  & \\
    & \gamma' \gamma \act x \ar[ddd]_(.4){\gamma' \act \Phi_{(\gamma,\chi)}(x)} \ar[rr]_(.4){\gamma'\gamma \act f} \ar[ul]^{\Phi_{(\gamma',\chi')}(\gamma \act x)}  \ar@{-->}[ddr] & & \gamma' \gamma \act y \ar[ul]_{\Phi_{(\gamma',\chi')}(\gamma \act y)} \ar[ddd]^{\gamma' \act \Phi_{(\gamma,\chi)}(y)} \\
    & & & \\
    \partial (\chi') \gamma' \partial (\chi) \gamma \act x \ar[rr]^(.6){\partial (\chi') \gamma' \partial (\chi) \gamma \act f} & & \partial (\chi') \gamma' \partial (\chi) \gamma \act y & \\
    & \gamma' \partial (\chi)\gamma \act x \ar[ul]^{\Phi_{(\gamma',\chi')}(\partial (\chi)\gamma \act x)} \ar[rr]_{\gamma'\partial (\chi)\gamma \act f} & & \gamma'\partial (\chi)\gamma \act y \ar[ul]_{\Phi_{(\gamma',\chi')}(\partial (\chi) \gamma \act y)}
  }
\end{equation}
We have omitted the images under $\Phi_{\gamma'}$ and $\Phi_{(\partial
  \chi')\gamma'}$ of the original diagonal $(\gamma,\chi) \act f$ for
clarity. However, they are the diagonals on the front and back face of
this cube. Together with the corresponding edges between these faces,
they form a square whose diagonal is the target we want, just as
before.

The target morphism of the square $\squaremor{ (\gamma',
  \chi'),(\gamma,\chi)\act f }$, hence also of our composite square,
is therefore the dotted arrow across the diagonal of this cube.

The source object of this morphism is $\gamma' \gamma \act x$, and its
target object is $\partial (\chi')\gamma' \partial (\chi) \gamma \act
y$.

However, we already found that the composite square was $\squaremor{
  (\gamma'\gamma, \chi'(\gamma'\rhd\chi)),f }$. Computing the target
of this square directly from (\ref{eq:triplecat-3cell}) shows that its
target morphism begins at $\gamma' \gamma \act x$ and ends at
$(\partial (\chi' (\gamma' \rhd \chi) ) \gamma' \gamma \act x$. A
straightforward application of the crossed module axioms confirms that
indeed:
\begin{equation}\label{eq:crossmod-target}
  \partial (\chi')\gamma' \partial (\chi) \gamma = \partial (\chi' (\gamma' \rhd \chi) ) \gamma' \gamma
\end{equation}
as it must be.

Knowing that we are in a double category where composition is
well-defined tells us more. Just as the naturality square at the
bottom of (\ref{eq:triplecat-3cell}) gave us two expressions for
$(\gamma,\chi) \act f$, the cube of naturality squares
(\ref{eq:compositesquare-target}) gives us a total of six expressions
which equal the morphism
\begin{equation}
  ( \gamma' \gamma , \partial (\chi' (\gamma' \rhd \chi) ) ) \act f
\end{equation}
As before, each is a composite of morphisms in $\C$, most directly
written as:
\begin{equation}\label{eq:equivalent-composites}
  \begin{array}{ccccccc}
    & & \Phi_{(\gamma',\chi')}(\partial (\chi) \gamma \act y) & \circ & (\gamma ' \act \Phi_{(\gamma,\chi)}(y)) & \circ & (\gamma' \gamma \act f) \\
    &=& (\partial (\chi') \gamma') \act \Phi_{(\gamma,\chi)}(y)) & \circ & \Phi_{(\gamma',\chi')}(\gamma \act y) & \circ & (\gamma' \gamma \act f ) \\
    &=& \Phi_{(\gamma',\chi')}(\partial (\chi)\gamma \act y) & \circ & (\gamma' \partial (\chi) \gamma \act f) & \circ & (\gamma' \act \Phi_{(\gamma,\chi)}(x))\\
    &=&  (\partial (\chi')\gamma' \partial (\chi) \gamma \act f) & \circ & \Phi_{(\gamma',\chi')}(\partial (\chi)\gamma \act x) & \circ & (\gamma' \act \Phi_{(\gamma,\chi)}(x)) \\
    &=&  (\partial (\chi')\gamma' \act \Phi_{(\gamma,\chi)}(y)) & \circ & (\partial (\chi') \gamma' \gamma \act f ) & \circ & \Phi_{(\gamma',\chi')}(\gamma \act x) \\
    &=& (\partial (\chi')\gamma' \partial (\chi) \gamma \act f) & \circ & (\partial (\chi')\gamma' \act \Phi_{(\gamma,\chi)}(x)) & \circ & \Phi_{(\gamma',\chi')}(\gamma \act x)
  \end{array}
\end{equation}
The fact that these are all equal is a consequence of the fact that
$\Phi$ is an action, hence every square in the cube is either a
naturality square or the image of a naturality square under a functor,
and therefore every square commutes. Some of the terms may also be
written differently in light of (\ref{eq:crossmod-target}) and similar
applications of crossed-module axioms.

In general, a composite of $n$ squares will give rise to many
different equivalent expressions for the target morphism, in the same
way. In particular, one can easily see by induction that one obtains
an $(n+1)$ cube by repeatedly applying $\Phi$ at each step. Counting
paths around such a cube, one finds $(n+1)!$ possible composites in
$\C$ which amount to the same morphism.

Clearly, using the 2-group composition as in
(\ref{eq:square-transcomp}) to find the target
(\ref{eq:square-transcomp-target1}) is more direct, and since we
know the transpose of $\C \wquot \G$ is a double category, we can use
the 2-group action to find the target directly.

\subsubsection{Horizontal and Vertical 2-Categories}

Finally, we turn to the double category point of view on $\C \wquot
\G$. When speaking of double categories, one often refers to various
smaller structures which arise from specialization. We already know
that the horizontal and vertical categories are, respectively, $\C$
itself, and $\C^{(0)} \wquot \G^{(0)}$. Moreover, there are also two
categories which share the same collection of morphisms (the squares
of the double category), but have different sets of objects (namely,
the horizontal and vertical morphisms). Again, we have already seen
that these are, respectively, $\G \times \C$ and $\C^{(1)} \wquot
\G^{(1)}$.

However, one also speaks of the horizontal and vertical 2-categories
of a double category. These are the last aspect of the structure of
$\C \wquot \G$ we will consider.

\begin{definition} The \textit{horizontal 2-category} $H({\cal D})$ of a double
  category ${\cal D}$ has the same objects and 1-morphisms as its
  horizontal category. The 2-morphisms consist only of the squares
  of ${\cal D}$ for which the vertical morphisms on the boundary are
  identities. The \textit{vertical 2-category}  $V({\cal D})$ is defined similarly
  (i.e. it is the horizontal 2-category of $\widetilde{\cal D}$).
\end{definition}

It is straightforward to check that these form 2-categories, but since
it is a standard fact, we will not verify all the details
here. However, we will consider the horizontal and vertical
2-categories for $\C \wquot \G$. It should be clear from what we have
just said that these will amount to $\C$ and $\C^{(0)} \wquot
\G^{(0)}$, extended by adding certain 2-morphisms.

\begin{theorem}\label{thm:horiztwocat}
  The horizontal 2-category $H(\C \wquot \G)$ of the transformation
  double category associated to a 2-group action has the same objects
  and 1-morphisms as $\C$. Given objects $x, y \in \C$, and a pair of
  morphisms $f,f' : x \ra y$, the 2-morphisms in $Hom(f, f')$ are
  labeled by $(1_G,\chi) : 1_G \Ra 1_G$ such that
  \begin{equation}\label{eq:horiz2mortarget}
    f' = f \circ \Phi_{(1_G,\chi)}(x)
  \end{equation}
  These 2-morphisms compose horizontally and vertically just as the
  endomorphisms of $1_G$ in $\G$.
\end{theorem}
\begin{proof}
  The objects and 1-morphisms are those of the horizontal category,
  which is just $\C$ by definition.

  The 2-morphisms correspond to squares of the form
  (\ref{eq:squaredef}) for which the vertical morphisms are
  identities. These must be of the form $\squaremor{(1_G,\chi),f}$,
  since the horizontal source of $\squaremor{(\gamma,\chi),f}$ is $(\gamma,x) = Id_x$, and thus $\gamma
  = 1_G$. Moreover, since the target $(\partial(\chi) \gamma, y) =
  (\partial(\chi),y) = Id_y$, we must also have that $\partial(\chi) =
  1_G$. So only 2-endomorphisms of the identity of $\G$ enter as
  2-morphisms.

  The vertical target is then the horizontal morphism 
  \begin{eqnarray}
    (1_G,\chi) \act f & = & (\partial(\chi)\act f) \circ \Phi_{(1_G,\chi)}(x) \\
    \nonumber & = & f \circ \Phi_{(1_G,\chi)}(x)
  \end{eqnarray}

  When $\gamma_1$ and $\gamma_2$ in (\ref{eq:squarehorizcomp}) and (\ref{eq:squarevertcomp}) are
both equal to $1_G$, these horizontal and vertical composition rules
   just amount
    to multiplication in $H$, that is, the composition of
    2-endomorphisms of the identity.
\end{proof}

Intuitively, this says that the 2-morphisms with source $f$ 
in $H(\C \wquot \G)$ correspond
exactly to the endomorphisms $(1_G,\chi)$ of $1_G$ in $\G$. In the
language of crossed modules, they correspond to $\chi \in
ker(\partial)$. The target of such a 2-morphism is found by precomposing $f$ with the
endomorphisms of $x$ given by the natural transformations
$\Phi_{(1_G,\chi)}$. Clearly, these labels $\chi$ for 2-morphisms
sourced at $f$ will be the same for all $f$. Depending on the precise
details of the action $\Phi$, all such 2-morphisms might be
endomorphisms of $f$ itself, or the targets might all be different
from $f$ and from each other. This is the type of information about
the action $\Phi$ which is captured by $H(\C \wquot \G)$.

\begin{theorem}\label{thm:verttwocat}
  The vertical 2-category $V(\C \wquot \G)$ of the transformation double
  category associated to a 2-group action has the same objects and
  1-morphisms as $\C^{(0)} \wquot \G^{(0)}$. Given an object $x \in
  \C$ and $\gamma,\gamma' \in G$ such that $\gamma \act x = \gamma'
  \act x$, there will be a 2-morphism from $(\gamma,x)$ to
  $(\gamma',x)$ for each $\chi$ such that $\Phi_{(\gamma,\chi)}(x) =
  Id_{\gamma \act x}$ and $\gamma'= \partial(\chi)\gamma$.
\end{theorem}
\begin{proof}
  The objects and 1-morphisms are those of the vertical category,
  which is $\C^{(0)} \wquot \G^{(0)}$ by Theorem
  \ref{thm:transposetransform}.

  The 2-morphisms correspond to squares of the form
  (\ref{eq:squaredef}) in which $x=y$ and the horizontal morphisms on
  the boundary are both identities. Since the vertical source is
  $f=Id_x$, they are squares of the form
  $\squaremor{(\gamma,\chi),Id_x}$.

  Moreover, the vertical target must also be the identity. Now, to
  begin with, this must mean that $\gamma \act x = \gamma' \act x$
  (taking $\gamma' = \partial(\chi)\gamma$), as stated in the
  theorem. This vertical target is the horizontal morphism
  \begin{eqnarray}
    (\gamma,\chi) \act f & = & (\gamma'\act Id_x) \circ \Phi_{(\gamma,\chi)}(x) \\
    \nonumber & = & Id_{\gamma' \act x} \circ \Phi_{(\gamma,\chi)}(x)
  \end{eqnarray}
  The second equality holds because $\gamma'$
  acts on $f = Id_x$ by the functor $\Phi_{\gamma'}$, and functors
  respect identities. Thus, this condition implies that not only are
  $\gamma \act x$ and $\gamma' \act x$ equal, but indeed the natural
  transformation $\Phi_{(\gamma,\chi)}$ intertwines them by the
  identity on $\gamma \act x$.

  The compositions again follow (\ref{eq:squarehorizcomp}) and
  (\ref{eq:squarevertcomp}).
\end{proof}

Intuitively, the transformation groupoid $\C^{(0)} \wquot \G^{(0)}$
displays the object part of the action $\hat{\Phi}$. The morphism part
can potentially appear here also, relating local symmetry
transformations which carry $x$ to the same place. However, as we have
seen, such a relation need not exist if $\gamma \neq \gamma '$, since
even though $\gamma \act x = \gamma' \act x$, there might be no
natural transformation intertwining them, or none intertwining them
with identities. Conversely, there may be more than one such natural
transformation. This is the information about the full action $\Phi$
captured by $V(\C \wquot \G)$ which the mere groupoid $\C^{(0)}
\wquot \G^{(0)}$ cannot see.

We would next like to explicitly describe the above, by examining the
structure of the transformation double groupoid $\G \wquot \G$,
associated to our example of the adjoint action of a 2-group $\G$ on
itself.

\subsection{The Transformation Double Category for the Adjoint
  Action}\label{sec:trans-adjoint}

As in all transformation double groupoids, the horizontal category of
$\G \wquot \G$ is just the same as $\G$, seen as a category 
(which happens to be monoidal).
The vertical morphisms and squares will always be labelled by pairs
of, respectively, objects and morphisms from $\G$.

Summing up the above, we have:

\begin{corollary}\label{thm:gmodg}
The transformation double groupoid $\G
\wquot \G$ for the adjoint action of a 2-group $\G$, classified by the
crossed module $(G,H,\rhd,\partial)$, has:
\begin{itemize}
\item Objects: labelled by $g \in G$
\item Horizontal Category: the underlying category of $\G$
\item Vertical Category: same as $G \wquot G$, the transformation
  groupoid for the adjoint action of $G$
\item Squares: labelled by pairs $((\gamma_1,\chi_1),(g_1,\eta_1)) \in
  (G \times H) \times (G \times H)$, denoted by $\squaremor{
    (\gamma_1, \chi_1),(g_1, \eta_1) }$. As in (\ref{eq:squaredef}),
  the horizontal and vertical source and target of this square are displayed as follows:
  \begin{equation}\label{eq:adjoint-action-square}
    \xymatrix@C=7.5pc@R=4pc {
      g_1 \ar[r]^{\eta_1} \ar[d]_{\Phi_{\gamma_1}} \drtwocell<\omit>{\omit *+[F]{(\gamma_1,\chi_1),(g_1,\eta_1)}}  &  g_2 \ar[d]^{\Phi_{\gamma_2}}  \\
      g_3 \ar[r]_{\eta_2}  & g_4
    }   
  \end{equation}
  where $\gamma_2=\partial(\chi_1)\gamma_1$. Here $\eta_2$ is determined by:
  \begin{equation}
    \xybiglabels \vcenter{\xymatrix@M=0pt@=3pc{\ar@{-} [d]  \ar@{-} [r]^{g_3} \ar@{} [dr]|{\eta_2} & \ar@{-} [d] \\ \ar@{-} [r]_{g_4}  & }}
    \,  = \,
    \xybiglabels \vcenter{\xymatrix @=3pc @W=0pc @M=0pc { \ar@{-}[r] ^{\gamma_1} \ar@{-}[d]_{} \ar@{}[dr]|{\chi_1} & \ar@{-}[r] ^{g_1} \ar@{-}[d]|{}
        \ar@{}[dr]|{\eta_1} & \ar@{-}[r] ^{\gamma_1^{-1}} \ar@{-}[d]^{}  \ar@{}[dr]|{\chi_1^{-h}} & \ar@{-}[d]^{}
        \\ \ar@{-}[r] _{\gamma_2} & \ar@{-}[r] _{g_2} & \ar@{-}[r] _{\gamma_2^{-1}} &
      }}.
  \end{equation}

\item Horizontal composition of squares is given by multiplication in
  $H$:
  \begin{equation}
    \squaremor{ (\gamma_2, \chi_2),(g_2, \eta_2)} 
    \, \circ_h \, \squaremor{ (\gamma_1, \chi_1),(g_1, \eta_1) }
    \, = \, \squaremor{ (\gamma_1, \chi_2\chi_1),(g_1, \eta_2\eta_1)}
  \end{equation}
  where $g_2=\partial(\eta_1)g_1$, and vertical composition of squares  is given by:
  \begin{eqnarray}    
 \squaremor{ (\gamma_1, \chi_1),(\gamma_3,\chi_2)\act (g_1,\eta_2) } 
  &  \circ_v &  \squaremor{ (\gamma_3, \chi_2),(g_1, \eta_2) }   \, = \, \nonumber\\
& & \nonumber\\
& &  \squaremor{ (\gamma_1\gamma_3, \chi_1(\gamma_1\rhd\chi_2)),(g_1, \eta_2) } 
  \end{eqnarray} 
  which makes them into the transformation groupoid $(G \ltimes H) \wquot (G \ltimes H)$,
\end{itemize}
\end{corollary}
\begin{proof}
  This is a direct consequence of the structure of $\C \wquot \G$ laid
  out in Definition \ref{def:action-double-gpd} and Theorem \ref{thm:trans-dbl-cat}, with the structure of
  $\G$ as a categorical group given explicitly for $\C$.
\end{proof}

Now we will find the horizontal and vertical 2-categories, $H(\G
\wquot \G)$ and $V(\G \wquot \G)$, of the double category associated
to the adjoint action of $\G$

\begin{corollary}
  If $\G$ is a 2-group classified by the crossed module
  $(G,H,\rhd,\partial)$, then the horizontal 2-category $H(\G \wquot
  \G)$ associated to the adjoint action of $\G$ has:
  \begin{itemize}
    \item \textbf{Objects}: $g \in G$
    \item \textbf{Morphisms}: $(g,\eta) \in G \ltimes H$ with source,
      target, and composition as in $\G$
    \item \textbf{2-Morphisms}: $((g,\eta),\chi) \in (G \ltimes H)
      \times \opname{ker}(\partial)$, whose source is $(g,\eta)$, and whose target is
      \begin{equation}\label{eq:adjointhoriztwocell}
        (g,\eta ') = ( g , \eta \cdot (\chi g \rhd \chi^{-1} ) )
      \end{equation}
      Composition is given by multiplication in the subgroup
      $\opname{ker}(\partial) \subset H$.
  \end{itemize}
\end{corollary}
\begin{proof}
  Taking the expression (\ref{eq:horiz2mortarget}) for the target of a
  2-morphism, in the case of the adjoint action we use the fact that
  $\Phi_{(1_G,\chi)}(g)$ is defined in
  (\ref{eq:adjoint-nat-trans}). Substituting this gives the target specified
  in (\ref{eq:adjointhoriztwocell}). Composition as multiplication in
  $H$ also follows from Theorem \ref{thm:horiztwocat}.
\end{proof}

This tells us how to extend $\G$ to a 2-category to reflect part of
the symmetry that arises from its action on itself. The vertical
2-category, by contrast, can be better seen as an extension of $G
\wquot G$, the transformation groupoid for the actions of just the
group of objects on itself. In particular, we have the following.

\begin{corollary}
    If $\G$ is a 2-group classified by the crossed module
    $(G,H,\rhd,\partial)$, then the vertical 2-category $V(\G \wquot
    \G)$ associated to the adjoint action of $\G$ has:
    \begin{itemize}
      \item \textbf{Objects}: $g \in G$
      \item \textbf{Morphisms}: $(\gamma,g) \in G \times G$ with
        source, target, and composition maps as in the transformation
        groupoid for the adjoint action of $G$ on itself
      \item \textbf{2-Morphisms}: There will be a 2-morphism from
        $(\gamma,g)$ to $(\gamma',g)$, for each $\chi \in H$ such that 
        \begin{equation}
          (\gamma g \gamma^{-1}) \rhd \chi^{-1} = \chi^{-1}
        \end{equation}
        (That is, $\chi^{-1}$ is a fixed point for $\gamma g
        \gamma^{-1}$ under the action $\rhd$).
    \end{itemize}
\end{corollary}
\begin{proof}
  In Theorem \ref{thm:verttwocat}, we established when there is a
  2-morphism in $V(\C \wquot \G)$. In our case, taking $x = g$, we find
  there is such a 2-morphism for any $\chi$ such that
  $\Phi_{(\gamma,\chi)}(g) = Id_{\gamma g \gamma^{-1}}$.

  But we have the expression (\ref{eq:adjoint-nat-trans}) for the maps
  defining the natural transformation $\Phi_{(\gamma,\chi)}$. So this
  condition says that
  \begin{equation}
    \chi (\gamma g \gamma^{-1}) \rhd \chi^{-1} = 1_H
  \end{equation}
  or in other words, that
  \begin{equation}
    (\gamma g \gamma^{-1}) \rhd \chi^{-1} = \chi^{-1}
  \end{equation}
  This is exactly the statement above.

  As usual, composition follows (\ref{eq:squarehorizcomp}) and
  (\ref{eq:squarevertcomp}).
\end{proof}

Intuitively, this tells us that the information about the action
$\Phi$ captured by the 2-morphisms relating two different symmetry
relations $\gamma$ and $\gamma '$ taking $g$ to $g$ is fundamentally
information about the crossed module structure, and the action $\rhd$
of $G$ on $H$.

We have included these calculations of the horizontal and vertical
2-categories, partly because these particular slices of the double
category $\G \wquot \G$ (or $\C \wquot \G$ generally) are interesting
in their own right. But we also include them to illustrate that many
squares in the double category do not appear in either the horizontal
or vertical 2-categories. Each is missing a great deal of information
about the action $\Phi$, and we need to go to the
double category $\G \wquot \G$ (or $\C \wquot \G$ in general)
for the full information.


\begin{references*}

\bibitem{hda5}
{\sc {Baez}, J.~C., and {Lauda}, A.~D.}
\newblock {Higher-dimensional algebra. V: 2-Groups.}
\newblock {\em {Theory Appl. Categ.} 12\/} (2004), 423--491.

\bibitem{baez-perez}
{\sc {Baez}, J.~C., and {Perez}, A.}
\newblock {Quantization of strings and branes coupled to BF theory.}
\newblock {\em {Adv. Theor. Math. Phys.} 11}, 3 (2007), 451--469.

\bibitem{basch-hgt}
{\sc {Baez}, J.~C., and {Schreiber}, U.}
\newblock {Higher gauge theory.}
\newblock In {\em {Categories in algebra, geometry and mathematical physics.
  Conference and workshop in honor of Ross Street's 60th birthday, Sydney and
  Canberra, Australia, July 11--16/July 18--21, 2005}}. Providence, RI:
  American Mathematical Society (AMS), 2007, pp.~7--30.

\bibitem{borceux}
{\sc Borceux, F.}
\newblock {\em Handbook of Categorical Algebra I}.
\newblock No.~50 in Encyclopedia of Mathematics and its Applications. Cambridge
  University Press, 1994.

\bibitem{brown-higgins-sivera}
{\sc {Brown}, R., {Higgins}, P.~J., and {Sivera}, R.}
\newblock {\em {Nonabelian algebraic topology. Filtered spaces, crossed
  complexes, cubical homotopy groupoids. With contributions by Christopher D.
  Wensley and Sergei V. Soloviev.}}
\newblock Z\"urich: European Mathematical Society (EMS), 2011.

\bibitem{brown-mackenzie}
{\sc {Brown}, R., and {Mackenzie}, K.~C.}
\newblock {Determination of a double Lie groupoid by its core diagram.}
\newblock {\em {J. Pure Appl. Algebra} 80}, 3 (1992), 237--272.

\bibitem{brownspencer}
{\sc {Brown}, R., and {Spencer}, C.~B.}
\newblock {Double groupoids and crossed modules.}
\newblock {\em {Cah. Topologie G\'eom. Diff\'er. Cat\'egoriques} 17\/} (1976),
  343--362.

\bibitem{brownspencer2}
{\sc {Brown}, R., and {Spencer}, C.~B.}
\newblock {$\mathcal G$-groupoids, crossed modules and the fundamental groupoid
  of a topological group.}
\newblock {\em {Nederl. Akad. Wet., Proc., Ser. A} 79\/} (1976), 296--302.

\bibitem{burciu-natale}
{\sc {Burciu}, S., and {Natale}, S.}
\newblock {Fusion rules of equivariantizations of fusion categories.}
\newblock {\em {J. Math. Phys.} 54}, 1 (2013), 013511, 21.

\bibitem{chenglauda}
{\sc Cheng, E., and Lauda, A.}
\newblock Higher-dimensional categories: An illustrated guidebook.
\newblock
  {\tt{http://www.math.uchicago.edu/~eugenia/guidebook/guidebook-new.pdf}},
  2004.

\bibitem{quintets}
{\sc {Ehresmann}, C.}
\newblock {Categorie double quintettes; applications covariantes.}
\newblock {\em {C. R. Acad. Sci., Paris} 256\/} (1963), 1891--1894.

\bibitem{ehresmann}
{\sc {Ehresmann}, C.}
\newblock {Categories doubles et categories structurees.}
\newblock {\em {C. R. Acad. Sci., Paris} 256\/} (1963), 1198--1201.

\bibitem{elgueta}
{\sc Elgueta, J.}
\newblock Permutation 2-groups {I}: structure and splitness.
\newblock {\em Adv. Math. 258\/} (2014), 286--350.

\bibitem{giraud}
{\sc {Giraud}, J.}
\newblock {Cohomologie non abelienne. (Non-Abelian cohomology).}
\newblock {Die Grundlehren der mathematischen Wissenschaften. Band 179.
  Berlin-Heidelberg-New York: Springer-Verlag.}, 1971.

\bibitem{kirillov}
{\sc {Kirillov jun.}, A.}
\newblock {Modular categories and orbifold models.}
\newblock {\em {Commun. Math. Phys.} 229}, 2 (2002), 309--335.

\bibitem{lack-bicat}
{\sc {Lack}, S.}
\newblock {A 2-categories companion.}
\newblock In {\em {Towards higher categories}}. Berlin: Springer, 2010,
  pp.~105--191.

\bibitem{leinster-bb}
{\sc Leinster, T.}
\newblock Basic bicategories.
\newblock {\tt{arXiv:math/9810017}}.

\bibitem{leinster}
{\sc {Leinster}, T.}
\newblock {\em {Higher operads, higher categories.}}
\newblock Cambridge: Cambridge University Press, 2004.

\bibitem{martinspickenii}
{\sc {Martins}, J.~F., and {Picken}, R.}
\newblock {Surface holonomy for non-Abelian 2-bundles via double groupoids.}
\newblock {\em {Adv. Math.} 226}, 4 (2011), 3309--3366.

\bibitem{martinspickeni}
{\sc {Martins}, J.~F., and {Picken}, R.}
\newblock {The fundamental Gray 3-groupoid of a smooth manifold and local
  3-dimensional holonomy based on a 2-crossed module.}
\newblock {\em {Differ. Geom. Appl.} 29}, 2 (2011), 179--206.

\bibitem{martinsporter}
{\sc {Martins}, J.~F., and {Porter}, T.}
\newblock {On Yetter's invariant and an extension of the Dijkgraaf-Witten
  invariant to categorical groups.}
\newblock {\em {Theory Appl. Categ.} 18\/} (2007), 118--150.

\bibitem{morton-etqft}
{\sc Morton, J.~C.}
\newblock Cohomological twisting of 2-linearization and extended {TQFT}.
\newblock {\em J. Homotopy Relat. Struct. 10}, 2 (2015), 127--187.

\bibitem{murray-gerbes}
{\sc {Murray}, M.}
\newblock {Bundle gerbes.}
\newblock {\em {J. Lond. Math. Soc., II. Ser.} 54}, 2 (1996), 403--416.

\bibitem{naidu}
{\sc {Naidu}, D.}
\newblock {Crossed pointed categories and their equivariantizations.}
\newblock {\em {Pac. J. Math.} 247}, 2 (2010), 477--496.

\bibitem{nlab-doublecat}
{\sc nCategories Lab~Wiki}.
\newblock Double category.
\newblock Retrieved Dec 19, 2013.

\bibitem{schreiber-sati-stasheff}
{\sc {Sati}, H., {Schreiber}, U., and {Stasheff}, J.}
\newblock {Twisted differential string and fivebrane structures.}
\newblock {\em {Commun. Math. Phys.} 315}, 1 (2012), 169--213.

\bibitem{schreiberwaldorfii}
{\sc {Schreiber}, U., and {Waldorf}, K.}
\newblock {Smooth functors vs. differential forms.}
\newblock {\em {Homology Homotopy Appl.} 13}, 1 (2011), 143--203.

\bibitem{weinstein-symmetry}
{\sc {Weinstein}, A.}
\newblock {Groupoids: unifying internal and external symmetry. A tour through
  some examples.}
\newblock In {\em {Groupoids in analysis, geometry, and physics. AMS-IMS-SIAM
  joint summer research conference, University of Colorado, Boulder, CO, USA,
  June 20--24, 1999}}. Providence, RI: American Mathematical Society (AMS),
  2001, pp.~1--19.

\bibitem{yettqft}
{\sc {Yetter}, D.~N.}
\newblock {Topological quantum field theories associated to finite groups and
  crossed $G$-sets.}
\newblock {\em {J. Knot Theory Ramifications} 1}, 1 (1992), 1--20.

\end{references*}

\end{document}